\documentclass[11pt]{amsart}

\usepackage{amsmath,amssymb}
\usepackage{graphicx}
\usepackage{float}
\usepackage{mathrsfs}
\usepackage{epstopdf}
\usepackage{graphicx}
\usepackage{amsthm}
\usepackage{color}
\usepackage[colorlinks,linkcolor=red,anchorcolor=red, citecolor=red]{hyperref}

\newtheorem{proposition}{Proposition}[section]
\newtheorem{theorem}{Theorem}[section]
\newtheorem{definition}{Definition}[section]
\newtheorem{lemma}{Lemma}[section]
\newtheorem{remark}{Remark}[section]

\makeatletter

\newcommand{\Rmnum}[1]{\expandafter\@slowromancap\romannumeral #1@}
\makeatother
\begin{document}

\title[]{Polynomial decay of the gap length for  $C^{k}$ quasi-periodic Schr\"{o}dinger operators and spectral application}

\author{Ao Cai}
\address{
	Chern Institute of Mathematics and LPMC, Nankai University, Tianjin 300071, China; and Departmento de Matem\'{a}tica and CMAFCIO, Faculdade de Ci\^{e}ncias, Universidade de Lisboa, Portugal.
}
\email{acai@fc.ul.pt; godcaiao@126.com}

\author {Xueyin Wang}
\address{
	Chern Institute of Mathematics and LPMC, Nankai University, Tianjin 300071, China.
}
\email{xueyinwang1995@163.com, 2889190407@qq.com}

\begin{abstract}
	
	For the quasi-periodic Schr\"{o}dinger operators in the local perturbative regime where the frequency is Diophantine and the potential is $C^k$ sufficiently small depending on the Diophantine constants, we prove that the length of the corresponding spectral gap has a polynomial decay upper bound with respect to its label. This is based on a refined quantitative reducibility theorem for $C^k$ quasi-periodic ${\rm SL}(2,\mathbb{R})$ cocycles, and also based on the Moser-P\"{o}schel argument for the related Schr\"{o}dinger cocycles. As an application, we are able to show the homogeneity of the spectrum.
		
\end{abstract}

\maketitle

\section{Introduction}
Consider the one dimensional discrete Schr\"{o}dinger operator defined on $\ell^{2}(\mathbb{Z})$:
\begin{equation*}
(H_{V,\alpha,\theta} u)_{n} = u_{n+1} + u_{n-1} + V(\theta+n\alpha)u_{n}, \ \ \forall \  n\in \mathbb{Z},
\end{equation*}
where $\theta\in \mathbb{T}^{d}=\mathbb{R}^{d}/\mathbb{Z}^{d}$ is called the phase, $\alpha\in \mathbb{T}^{d}$ is the  frequency, and $V:\mathbb{T}^{d}\rightarrow \mathbb{R}$ is called the potential.  The most important example is the almost Mathieu operator (AMO), which is defined as
\begin{equation*}
(H_{\lambda,\alpha,\theta}u)_{n} = u_{n+1}+u_{n-1}+2\lambda\cos2\pi(\theta+n\alpha)u_{n}, \ \ \forall \ n\in \mathbb{Z},
\end{equation*}
where $\lambda\in \mathbb{R}$ is called the coupling constant.

These operators have been widely studied due to their close relations to quasi-crystal and quantum Hall effect in physics, see \cite{Jito, Ds}. For Schr\"{o}dinger operators, people are always concerned with the topological structure of the spectrum and the property of the spectral measure. Note that there has been a lot of progress for the analytic potential $V\in C^{\omega}(\mathbb{T}^{d},\mathbb{R})$, see \cite{Av1, Av2, Le, Eli, Dam2}, however the results are relatively rare for the lower regularity case. In this paper, we focus on the  topological structure of the spectrum for Schr\"{o}dinger operators with $V\in C^{k}(\mathbb{T}^{d},\mathbb{R})$.

\subsection{Spectral gaps}
Denote by $\Sigma_{V,\alpha,\theta}$ the spectrum of $H_{V,\alpha,\theta}$, which is independent of $\theta$ if $(1, \alpha)$ is rationally independent.
It is well known that when $V$ is bounded, $H_{V,\alpha,\theta}$ is a bounded self-adjoint operator on $\ell^{2}(\mathbb{Z})$ and $\Sigma_{V,\alpha}\in \mathbb{R}$. Any bounded connected component of $\mathbb{R}\backslash \Sigma_{V,\alpha}$ is called a spectral gap.
By the Gap-Labelling Theorem \cite{De,JM}, each spectral gap can be labelled by a unique $m\in\mathbb{Z}^{d}$ such that $N_{V,\alpha}(E)=\langle m,\alpha \rangle \mod \mathbb{Z}$ (the  label should be $-m$ for the positive Laplacian case), where $N_{V,\alpha}(E)$ is the integrated density of states of $H_{V,\alpha,\theta}$.
Moreover, different gaps correspond to different labels.
Denote by $G_{m}(V)=(E_{m}^{-},E_{m}^{+})$ the gap with label $m$.
Recall that $\alpha$ is  Diophantine if $\alpha \in {\rm DC}_{d}(\gamma,\tau)$ with $\gamma>0, \tau>d$, where
\begin{equation*}
{\rm DC}_{d}(\gamma,\tau) =\left\{\alpha\in \mathbb{R}^{d}:\inf_{l\in \mathbb{Z}}|\langle m,\alpha \rangle-l|\geq  \frac{\gamma}{|m|^{\tau}}, \forall \ m\in\mathbb{Z}^{d}\backslash\{0\} \right\}.\footnote{${\rm DC}_{d} =\cup_{\gamma>0,\tau>d}{\rm DC}_{d}(\gamma,\tau)$ is of full Lebesgue measure. Actually, for any fixed $\tau>d$, $\cup_{\gamma>0}{\rm DC}_{d}(\gamma,\tau)$ is also of full Lebesgue measure.}
\end{equation*}

Seeds of the upsurge in studying the spectral structure had already been planted by D. Hofstadter in 1976, when he discovered the marvelous ``Hofstadter's butterfly" and gave a graphical representation of the spectrum of the AMO for $\lambda=1$ at different frequencies \cite{Hofs}.
Afterwards, Thouless-Kohmoto-Nightingale-Nijs \cite{Th} showed that the wings (namely gaps) of the butterfly are characterized by the Chern numbers (namely the labels \textquotedblleft $m$\textquotedblright \ defined above).
More recently, Simon \cite{Si} stated the conjecture, known as the famous \textquotedblleft Ten Martini Problem\textquotedblright \ after an offer by Mark Kac in 1981, that the spectrum of AMO is a Cantor set for all $\lambda\neq0$ and all $\alpha\in \mathbb{R}\backslash\mathbb{Q}$. The first contribution was made by Bellissard and Simon, in \cite{BS} they proved that for generic parameters $(\lambda,\alpha)$ the spectrum of AMO is a Cantor set through rational approximation.
Recent breakthrough belongs to Puig \cite{Puig} who proved that the Ten Martini Problem holds for all $\lambda\neq 0,\pm 1$ and $\alpha\in {\rm DC}_{1}$ via reducibility and Aubry duality, where $\alpha$ is of full Lebesgue measure.
Finally, Avila and Jitomirskaya \cite{Av1}  completely solved  this conjecture by several ingredients including Kotani theory and the analytic continuation techniques in the study of Wely's $m$-functions.
The \textquotedblleft Dry Ten Martini Problem\textquotedblright \ further conjectures that all the spectral gaps are open for all $\lambda \neq 0$ and all $\alpha\in \mathbb{R}\backslash\mathbb{Q}$.
Avila-Jitomirskaya \cite{Av1} proved that all the gaps of $H_{\lambda,\alpha,\theta}$ are open for $\beta(\alpha)>0$\footnote{Let $p_{n}/q_{n} \in \mathbb{Q}$ be the continued fraction approximants of $\alpha\in \mathbb{R}\backslash\mathbb{Q}$, then $\beta(\alpha) := \limsup_{n\rightarrow+\infty} \frac{\ln q_{n+1}}{q_{n}}$.} and $e^{-\beta(\alpha)}<|\lambda|<e^{\beta(\alpha)}$.
For $\beta(\alpha)<\infty$, Liu-Yuan \cite{LiuYuan} showed that $H_{\lambda,\alpha,\theta}$ has all gaps open if  $0<\lvert \lambda \rvert<e^{-C\beta}$ for some absolute constant $C$ by quantitative version Aubry duality.
For $\lambda \neq 1$, it was solved by Avila-You-Zhou \cite{Av2} via quantitative almost reducibility and Aubry duality.
Moreover, for $\alpha\in {\rm DC}_{1}$ Leguil-You-Zhao-Zhou \cite{Le} obtained the exponential asymptotic behavior on the gaps of AMO recently.
More precisely, for any  $m\in \mathbb{Z} \backslash\{0\}$ and $\xi\in(0,1)$, $\tilde{C}\lambda^{\tilde{\xi}|m|}\leq |G_{m}(\lambda)|\leq C\lambda^{\xi|m|}$ when $\lambda\in (0,1)$ and $\tilde{C}\lambda^{-\tilde{\xi}|m|}\leq|G_{m}(\lambda)|\leq C\lambda^{-\xi|m|}$ when $\lambda\in (1,\infty)$  with some constant $C=C(\lambda,\alpha, \xi)>0$, $\tilde{C}=\tilde{C}(\lambda,\alpha)$ and $\tilde{\xi}>1$.
For the  exponential decay of the gap length for the  extended Harper's model,  see Shi and Yuan \cite{ShiYuan} and Xu and Zhao \cite{Xuxu}.

For the analytic Schr\"{o}dinger operators,
Eliasson \cite{Eli} proved that for fixed $\alpha \in{\rm DC}_{d}$, $\Sigma_{V,\alpha}$ is a Cantor set for generic potentials by Moser-P\"{o}schel argument \cite{Mos}. Later, with the schemes of localization and Aubry duality,  Puig \cite{Puig2} extended this result to the non-perturbative case, i.e.,  the smallness of the potential is independent of the frequency. By almost reducibility,
Amor \cite{Am} proved that if $V$ is sufficiently small and $\alpha \in {\rm DC}_{d}$, then for all $m\in \mathbb{Z}^{d}$, $G_{m}(V)$ is at least sub-exponentially small with respect to $m$. After that, through multi-scale analysis scheme and Aubry duality Damanik and Goldstein \cite{Dam} showed that if
the Fourier coefficients of $V$ satisfy $|\widehat{V}(m)| \leq \varepsilon e^{-r_{0}|m|}$ for any $m\in\mathbb{Z}^{d}$,
then $|G_{m}(V)|\leq 2\varepsilon e^{-\frac{r_{0}}{2}|m|}$, where $\varepsilon:=\sup_{|\Im \theta|<r_{0}}|V(\theta)|$ is small enough. Later on, Leguil-You-Zhao-Zhou \cite{Le}  improved the results so that $|G_{m}(V)|\leq \varepsilon^{\frac{2}{3}}e^{-r|m|}$  holds for any $m\in \mathbb{Z}^{d}\backslash\{0\}$ if $\sup_{|\Im \theta|<r_{0}}|V(\theta)|<\varepsilon$ is sufficiently small, where the exponential decay rate $r\in (0,r_{0})$ can be arbitrarily close to $r_{0}$. For weakly coupled quasi-periodic Schr\"{o}dinger operators with Liouville frequencies, Liu and Shi \cite{LiuShi} proved the size of the spectral gaps decays exponentially.

Progresses have also been made in  Schr\"{o}dinger operators with potentials of lower regularity.
Avila-Bochi-Damanik
\cite{Av2} proved that given any frequency $\alpha\in \mathbb{T}^{d}$ rationally independent, the spectrum $\Sigma_{V,\alpha}$ is a Cantor set for
generic $V\in C^{0}(\mathbb{T}^{d},\mathbb{R})$.
Moreover, Cai and Ge \cite{Cai} proved Cantor spectrum for generic small and finitely smooth potential for $\alpha \in {\rm DC}_{d}$ by reducibility.
Wang and Zhang \cite{Wang} proved that if $V$ has two non-degenerate extremals (one is minimal and the other is maximal) and $V\in C^{2}(\mathbb{T},\mathbb{R})$, then $\Sigma_{\lambda V,\alpha}$ is a Cantor set for sufficiently large $\lambda$ and $\alpha \in {\rm DC}_{1}$.

As shown in \cite{Dam}, the decay rate of gap length is  closely related to the decay of Fourier coefficients of the potential $V$, which implies the dependence on the regularity of the potential.
Hence one can not expect exponential decay in the $C^{k}$ case, but only the polynomial decay since the Fourier coefficients of $C^{k}$ potentials decay polynomially. In this paper, we prove:
\begin{theorem} \label{main}
	Let $\alpha \in {\rm DC}_{d}(\gamma, \tau)$, $V\in C^{k}(\mathbb{T}^{d},\mathbb{R})$ with $k\geq D_{0}\tau$ where $D_{0}$ is a numerical constant. There exists $\varepsilon=\varepsilon(\gamma,\tau,k,d)>0$ such that if $\|V\|_{k}\leq\varepsilon$,  then
	\begin{equation}
	|G_{m}(V)| \leq \varepsilon^{\frac{1}{4}} |m|^{-\frac{k}{9}}.
	\end{equation}
\end{theorem}

\subsection{Homogeneous spectrum}
Based on polynomial
decay of the gap length and H\"{o}lder continuity of the integrated density of
states (see Section \ref{IDS}), one can easily conclude the homogeneity of the spectrum.
Recall that in \cite{Carl}, the concept of a homogeneous set is defined as follows:
\begin{definition}
	Let $\mu>0$, a closed set $\mathfrak{B}\subset \mathbb{R}$ is called $\mu$-homogeneous if
	\begin{equation*}
	|\mathfrak{B}\cap(E-\epsilon,E+\epsilon)|>\mu \epsilon, \ \ \ \forall \ E\in \mathfrak{B}, \ \forall \ 0<\epsilon<{\rm diam}  \mathfrak{B}.
	\end{equation*}
\end{definition}
The homogeneity of the spectrum is vital in the inverse spectral theory, see the fundamental work of Sodin and Yuditskii \cite{So1,So2}. It was shown that the homogeneity of the spectrum implies the almost periodicity of the associated potentials \cite{So1, Dam2}. Particularly, the homogeneity of the spectrum is deeply related to Deift's conjecture \cite{Bin,Dam3}.
Recall that Deift's conjecture asks: whether the solutions of the KdV equation are almost periodic if the initial data is almost periodic? In the continuous case,  Binder-Damanik-Goldstein-Lukic \cite{Bin} proved that for
small analytic quasi-periodic initial data with
Diophantine frequency, the solution of the KdV equation is almost periodic in time variable. In the  discrete case, Leguil-You-Zhao-Zhou \cite{Le} proved that for the subcritical potential $V\in C^{\omega}(\mathbb{T},\mathbb{R})$, the  Toda flow is almost periodic in time variable for the almost periodic intial data with $\beta(\alpha)=0$.

There are several positive results on the homogeneity of the spectrum in the analytic case. In the discrete case, Leguil-You-Zhao-Zhou \cite{Le} proved that if $\alpha\in {\rm SDC}_{1}$\footnote{If there exist $\gamma,\tau>0$ such that $\|n\alpha\|_{\mathbb{R}\backslash\mathbb{Z}}\geq \frac{\gamma}{|n|(\log |n|)^{\tau}}$, then $\alpha $ is strong Diophantine, and denote by $\alpha \in {\rm SDC}_{1}(\gamma,\tau)$. Let ${\rm SDC}_{1}=\cup_{\gamma,\tau >0} {\rm SDC}_{1}(\gamma,\tau)$.}, then for a (measure-theoretically) typical analytic potential $V\in C^{\omega}(\mathbb{T},\mathbb{R})$, $\Sigma_{V,\alpha}$ is $\mu$-homogeneous for some $\mu\in (0,1)$. For the special example AMO, they further proved that if $\beta(\alpha)=0$ and $\lambda\neq 1 $, then $\Sigma_{\lambda,_{\alpha}}$ is $\mu$-homogeneous.
Recently, Jian and Shi \cite{JianShi} proved the homogeneity of
the spectrum for the non-self dual extended Harper's model with a Liouville frequency.
Liu and Shi \cite{LiuShi} proved similar results for the weakly coupled quasi-periodic Schrodinger operators with Liouville frequencies.
In the continuous case, consider the continuous quasi-periodic Schr\"{o}dinger operators on $L^{2}(\mathbb{R})$:
\begin{equation*}
(\mathcal{L}_{V,\omega}y)(t) =-y''(t)+V(\omega t)y(t).
\end{equation*}
Damanik-Goldstein-Lukic \cite{Dam2} proved that for fixed $\omega\in {\rm DC}_{d}$, the spectrum of $\mathcal{L}_{V,\omega}$ is $\frac{1}{2}$-homogeneous if $V$ is analytic and small enough.

Lately, there is also an intriguing counter example constructed by Avila-Last-Shamis-Zhou. In \cite{Av3}, they showed that even for the AMO, its spectrum is not homogeneous if $e^{-\frac{2}{3}\beta(\alpha)}<\lambda<e^{\frac{2}{3}\beta(\alpha)}$. Note that all the related contributions above deal with the analytic case. However, to the best of our knowledge there is no result in the finitely differentiable case at present.

In this paper, we prove the following:
\begin{theorem}
	Let $\alpha \in {\rm DC}_{d}(\gamma,\tau)$, $V\in C^{k}(\mathbb{T}^{d},\mathbb{R})$ with $k\geq D_{0}\tau$, where $D_{0}$ is a numerical constant. There exists $\bar{\varepsilon}=\bar{\varepsilon}(\gamma,\tau,k,d)$ such that if $\|V\|_{k}\leq \bar{\varepsilon}$,  then $\Sigma_{V,\alpha}$ is $\mu$-homogeneous for some $\mu\in (0,1)$.
\end{theorem}

As is already mentioned in the abstract, our philosophy throughout the whole paper is that fine reducibility properties of the dynamics indicate nice spectral applications of the operator. In this spirit, the main body is organized as follows.

\subsection{Structure of the paper}
In Section 2, we give some basic concepts and notations, which will appear both in the dynamical side and the spectral side. In Section 3, we  prove the quantitative reducibility of $C^{k}$ quasi-periodic ${\rm SL}(2,\mathbb{R})$ cocycles with the rotation number being rational with respect to the frequency in the local perturbative regime.
As for Section 4, by the Moser-P\"{o}schel argument of quasi-periodic Schr\"{o}dinger cocycles and the reducibility results, we prove that the upper bound of the gap length decays polynomially. As an application, we show the homogeneity of the spectrum in Section 5.

\section{Preliminaries}
For a bounded analytic function $F$ (possibly matrix-valued) defined on $\mathcal{S}_{r}=\{\theta : \theta=(\theta_{1},\cdots,\theta_{d})\in \mathbb{C}^{d},|\Im\theta_{i}|<r, \forall \ i=1,\cdots,d\}$, let $|F|_{r}=\sup_{\theta\in \mathcal{S}_{r}} \|F(\theta)\|$ and denote by $C^{\omega}_{r}(\mathbb{T}^{d},*)$ the set of these $*$-value functions ($*$ will usually denote $\mathbb{R}, {\rm sl}(2,\mathbb{R})$ or ${\rm SL}(2,\mathbb{R})$). We also denote the set $C^{k}(\mathbb{T}^{d},*)$ to be the space of $k$ times differentiable with continuous $k$-th derivatives functions, endowed with the norm
\begin{equation*}
\|F\|_{k} :=\sup_{|k'|\leq k, \theta\in \mathbb{T}^{d}} \|\partial^{k'} F(\theta)\|.
\end{equation*}
In particular,
\begin{equation*}
\|F\|_{0} := \|F\|_{\mathbb{T}^{d}} = \sup_{\theta\in \mathbb{T}^{d}}\|F(\theta)\|.
\end{equation*}

\subsection{Quasi-periodic cocycle, uniform hyperbolicity and reducibility}
Given $A\in C^{0}(\mathbb{T}^{d},{\rm SL}(2,\mathbb{C}))$ and $(1,\alpha)$  rationally independent,  one can define the quasi-periodic cocycle $(\alpha, A)$:
\begin{equation*}
(\alpha, A):
\left\{
\begin{split}
&\ \ \ \mathbb{T}^{d}\times \mathbb{C}^{2} \rightarrow\mathbb{T}^{d}\times \mathbb{C}^{2},\\
&(\theta,v)\mapsto (\theta+\alpha, A(\theta)\cdot v).
\end{split}
\right.
\end{equation*}
The iterates of $(\alpha,A)$ are of the form $(\alpha,A)^{n} = (n\alpha, A_{n})$, where
\begin{equation*}
A_{n}(\theta):=
\begin{cases}
A(\theta+(n-1)\alpha) \cdots A(\theta+\alpha)A(\theta),&n\geq 0,\\
A^{-1}(\theta+n\alpha)A^{-1}(\theta+(n+1)\alpha)\cdots A^{-1}(\theta-\alpha),&n<0.
\end{cases}
\end{equation*}
We say the cocycle $(\alpha,A)$ is uniformly hyperbolic if for every $\theta\in \mathbb{T}^{d}$, there exists a continuous splitting $\mathbb{C}^{2}=E^{s}(\theta)\oplus E^{u}(\theta)$ such that for every $n\geq 0$,
\begin{equation*}
\begin{split}
\|A_{n}(\theta)v\| &\leq Ce^{-cn}\|v\|, \ \ v\in E^{s}(\theta),\\
\|A_{-n}(\theta)v\| &\leq Ce^{-cn}\|v\|, \ \ v\in E^{u}(\theta),
\end{split}
\end{equation*}
for some constants $C,c>0$. And the splitting is invariant by the dynamics:
\begin{equation*}
\begin{split}
A(\theta)E^{s}(\theta) &= E^{s}(\theta+\alpha),  \ \ \forall \ \theta \in \mathbb{T}^{d}, \\
A(\theta)E^{u}(\theta) &= E^{u}(\theta+\alpha),  \ \ \forall \ \theta \in \mathbb{T}^{d}.
\end{split}
\end{equation*}

Typical examples of ${\rm SL}(2,\mathbb{R})$ cocycles  are the Schr\"{o}dinger cocycles $(\alpha, S^{V}_{E})$:
\begin{equation*}
A(\theta) = S^{V}_{E}(\theta) = \begin{pmatrix}
E-V(\theta)&-1\\
1&0
\end{pmatrix}, \ \ \ E\in \mathbb{R}.
\end{equation*}
Those cocycles come from the eigenvalue equation of one dimensional quasi-periodic Schr\"{o}dinger operators on $\ell^{2}(\mathbb{Z})$:
\begin{equation*}
(H_{V,\alpha,\theta} u)_{n}=u_{n+1} + u_{n-1}+V(\theta+n\alpha)u_{n} = Eu_{n},
\end{equation*}
and any formal solution $u=(u_{n})_{n\in \mathbb{Z}}$ of $H_{V,\alpha,\theta}u=Eu$ satisfies
\begin{equation*}
\begin{pmatrix}
u_{n+1}\\
u_{n}
\end{pmatrix}=S^{V}_{E}(\theta+n\alpha)
\begin{pmatrix}
u_{n}\\
u_{n-1}
\end{pmatrix}, \ \ \forall \ n\in \mathbb{Z}.
\end{equation*}
The spectral properties of $H_{V,\alpha,\theta}$ and the dynamics of $(\alpha, S^{V}_{E})$ are closely related by the fact: $E\notin {\Sigma_{V,\alpha}}$ if and only if $(\alpha, S^{V}_{E})$ is uniformly hyperbolic \cite{Jo}.

The concepts of reducibility  and almost reducibility in both $C^{\omega}$ case and $C^{k}$ case are indispensable when one deals with Schr\"{o}dinger cocycles.
\begin{definition}\label{de1}
	The cocycle $(\alpha, A(\theta)) \in \mathbb{T}^{d}\times C^{\omega}(\mathbb{T}^{d}, {\rm SL}(2,\mathbb{R}))$ is called $C^{\omega}_{r,r'}$ almost reducible with $r'< r$ if there exist $B_{j}\in C^{\omega}_{r_{j}}(2\mathbb{T}^{d},{\rm SL}(2,\mathbb{R}))$, $A_{j}\in {\rm SL}(2,\mathbb{R})$ and $f_{j}\in C^{\omega}_{r_{j}}(\mathbb{T}^{d},{\rm sl}(2,\mathbb{R}))$ such that
	\begin{equation*}
	B_{j}(\theta+\alpha)^{-1}A(\theta)B_{j}(\theta) =A_{j}e^{f_{j}(\theta)},
	\end{equation*}
	and
	\begin{equation*}
	|f_{j}(\theta)|_{r_{j}}\rightarrow 0, \  \ j\rightarrow +\infty.
	\end{equation*}
	Moreover, the cocycle $(\alpha, A(\theta))$ is called $C^{\omega}_{r,r'}$ reducible if there exist $\tilde{B}\in C^{\omega}_{r'}(2\mathbb{T}^{d},{\rm SL}(2,\mathbb{R}))$ and $\tilde{A}\in {\rm SL}(2,\mathbb{R})$ such that
	\begin{equation*}
	\tilde{B}(\theta+\alpha)^{-1}A(\theta)\tilde{B}(\theta)=\tilde{A}.
	\end{equation*}
\end{definition}

In the finitely differentiable case, one has the definition
of almost reducibility and reducibility similarly as in Definition \ref{de1}. However, to avoid repeating the narrative, we give another equivalent definition.

\begin{definition}
	The cocycle $(\alpha, A(\theta))\in \mathbb{T}^{d}\times C^{k}(\mathbb{T}^{d}, {\rm SL}(2,\mathbb{R}))$ is called $C^{k,k'}$ almost reducible if the $C^{k'}$-closure of the its $C^{k'}$ conjugacies contains a constant.
	Moreover, the cocycle $(\alpha, A(\theta))$ is called reducible if its $C^{k'}$ conjugacies contain a constant.
\end{definition}

\subsection{Fibered rotation number }
\label{FRN}
Assume that $A\in C^{0}(\mathbb{T}^{d},{\rm SL}(2,\mathbb{R}))$ is homotopic to the identity. It induces the projective skew-product $F_{A}:\mathbb{T}^{d}\times \mathbb{S}^{1}\rightarrow\mathbb{T}^{d}\times \mathbb{S}^{1}$ with
\begin{equation*}
F_{A}(x,w):=\left(x+\alpha, \frac{A(x)\cdot w}{\|A(x)\cdot w\|}\right),
\end{equation*}
which is also homotopic to the identity. Thus we can lift $F_{A}$ to a map $\widetilde{F}_{A}:\mathbb{T}^{d}\times \mathbb{R}\rightarrow \mathbb{T}^{d}\times \mathbb{R}$ of form $\widetilde{F}_{A}(x,y) = (x+\alpha,y+\psi_{x}(y))$, where for every $x\in \mathbb{T}^{d}$, $\psi_{x}$ is $\mathbb{Z}$-periodic. The map $\psi:\mathbb{T}^{d}\times \mathbb{R}\rightarrow \mathbb{R}$ is called a lift of $A$. Let $\mu $ be any probability measure on $\mathbb{T}^{d}\times \mathbb{R}$ which is invariant by $\widetilde{F}_{A}$, and whose projection on the first coordinate is given by Lebesgue measure. The number
\begin{equation*}
\rho(\alpha, A) = \int_{\mathbb{T}^{d}\times \mathbb{R}} \psi_{x}(y)d\mu(x,y) \mod \mathbb{Z}
\end{equation*}
depends neither on the lift $\psi$ nor on the measure $\mu$, and is called the fibered rotation number of the cocycle $(\alpha,A)$.

Given $\varphi \in \mathbb{T}$, let $R_{\varphi}:=\begin{pmatrix}
\cos 2\pi \varphi &-\sin 2\pi \varphi  \\
\sin2\pi \varphi&\cos 2\pi\varphi
\end{pmatrix}$.
If $A:2\mathbb{T}^{d}\rightarrow{\rm SL}(2,\mathbb{R})$ is homotopic to $\theta\in \mathbb{T}^{d} \rightarrow R_{\frac{\langle n, \theta\rangle}{2}}$  for some $n\in \mathbb{Z}^{d}$, then  $n$ is called the degree of $A$ which is denoted by $\deg A$.
The fibered rotation number is invariant under real conjugacies which are homtopic to the identity. More generally, if $(\alpha,A_{1})$ is conjugated to $(\alpha, A_{2})$, i.e., $B(\theta+\alpha)^{-1}A_{1}(\theta)B(\theta) = A_{2}(\theta)$, for some $B: 2\mathbb{T}^{d}\rightarrow {\rm SL}(2,\mathbb{R})$, then
\begin{equation}
\rho(\alpha, A_{2}) = \rho(\alpha, A_{1})-\frac{\langle \deg B, \alpha \rangle}{2} \mod \mathbb{Z}. \label{rho}
\end{equation}
Moveover, it follows from the definition of rotation number that
\begin{lemma}\label{eig}
	If $A:\mathbb{T}^{d}\rightarrow {\rm SL}(2,\mathbb{R})$ is homotopic to the identity, then
	\begin{equation*}
	|\rho(\alpha,A)-\varphi|<\|A-R_{\varphi}\|_{\mathbb{T}^{d}}.
	\end{equation*}
\end{lemma}

\subsection{Integrated density of states}
 \label{IDS}

For Schr\"{o}dinger operators $H_{V,\alpha,\theta}$, an important concept is the  integrated density of states (IDS),
which is the function $N_{V,\alpha}: \mathbb{R}\rightarrow [0,1]$ defined by
\[
N_{V,\alpha}(E) = \int_{\mathbb{T}^{d}}\mu_{V,\alpha,\theta}(-\infty,E]d\theta,
\]
where  $\mu_{V,\alpha,\theta} = \mu_{V,\alpha,\theta}^{e_{-1}}+\mu_{V,\alpha,\theta}^{e_{0}}$ is the spectral measure of $H_{V,\alpha,\theta}$, and  $\{e_{i}\}_{i\in \mathbb{Z}}$ is the cannonical basis of $\ell^{2}(\mathbb{Z})$. In particular, $\{e_{-1}, e_{0}\}$ are called the
cyclic vectors of $H_{V,\alpha,\theta}$.

Another way to characterize the IDS is to calculate the distribution of eigenvalues by truncating the operator $H_{V,\alpha,\theta}$ in the interval $[-L, L]$.
Consider $H_{V,\alpha,\theta}^{L}$ the restriction of $H_{V,\alpha,\theta}$ in $[-L, L]$ with zero boundary conditions, and let
\begin{equation*}
N_{V,\alpha,\theta}^{L}(E) = \frac{1}{2L+1} \# \{ x: x \leq E, x \  \text{is eigenvalue of} \ H_{V,\alpha,\theta}^{L}\}.
\end{equation*}
Then the IDS can be defined by
\begin{equation*}
N_{V,\alpha}(E) = \lim_{L\rightarrow \infty} N_{V,\alpha,\theta}^{L}(E),
\end{equation*}
where the limit exists and is independent of $\theta$. For more details, readers can refer to \cite{AS}.
Moreover, $\rho(\alpha,S^{V}_{E})$ relates to the IDS as follows:
\begin{equation}
N_{V,\alpha}(E)=1-2\rho(\alpha,S^{V}_{E})  \mod \mathbb{Z}. \label{ids}
\end{equation}

\subsection{Analytic approximation}
Assume $f\in C^{k}(\mathbb{T}^{d},{\rm sl(2,\mathbb{R})})$, according to Zehnder's result \cite{Zen}, there exist a sequence $\{
f_{j}\}_{j\geq 1}$ with $f_{j}\in C^{\omega}_{\frac{1}{j}}(\mathbb{T}^{d},{\rm sl}(2,\mathbb{R}))$ and a universal constant $C^{\prime}>0$, such that
\begin{equation}
\begin{split}
&\|f_{j}-f\|_{k}\rightarrow 0,\ \ \ \ j\rightarrow +\infty,\\
&|f_{j}|_{\frac{1}{j}}\leq  C^{\prime}\|f\|_{k},\\
&|f_{j+1}-f_{j}|_{\frac{1}{j+1}}\leq C^{\prime}(j)^{-k}\|f\|_{k}.
\end{split}    \label{aa}
\end{equation}
Moreover, if $k\leq k'$ and $f\in C^{k'}(\mathbb{T}^{d},{\rm sl}(2,\mathbb{R}))$, these inequalities (\ref{aa}) hold with $k'$ instead of $k$. That means this sequence is obtained from $f$ regardless of its regularity. Actually, $f_{j}$ is the convolution of
$f$ with a map which does not depend on $k$.

\section{Dynamical estimates of $C^{k}$ quasi-periodic cocycle }
Kolmogorov-Arnold-Moser theory naturally arrives in the study of perturbative problems. In the light of classical KAM theory for smooth Hamiltonians, we use the analytic approximation to derive estimates of finitely differentiable cocycle from those of analytic ones. To achieve this, we first establish an analytic KAM theorem.

\subsection{Analytic KAM theorem}
Let us consider the following quasi-periodic ${\rm SL}(2,\mathbb{R})$ cocycle
\begin{equation*}
(\alpha, Ae^{f(\theta)}):
\left\{
\begin{split}
&\ \ \ \mathbb{T}^{d}\times\mathbb{R}^{2}\rightarrow \mathbb{T}^{d}\times\mathbb{R}^{2},\\
&(\theta,x)\mapsto(\theta+\alpha, Ae^{f(\theta)}\cdot x),
\end{split}
\right.
\end{equation*}
where $\alpha \in {\rm DC}_{d}(\gamma,\tau)$, $A\in {\rm SL}(2,\mathbb{R})$ and $f(\theta)\in C^{\omega}_{r}(\mathbb{T}^{d},{\rm sl}(2,\mathbb{R}))$ with $r>0$. Assume $|f|_{r}$ is sufficiently small, we are going to show that the perturbation falls into a much smaller magnitude by analytic conjugacy.

\begin{proposition}[\cite{Cai2}]	\label{pr1}
	Let $\alpha \in {\rm DC}_{d}(\gamma,\tau)$, $\gamma,r>0, \tau >d$ and $\sigma=\frac{1}{10}$. Then for any $r_{+}\in (0,r)$, there exist $c=c(\gamma,\tau,d)$ and a numerical constant $D$ such that if
	\begin{equation}
	|f|_{r}<\varepsilon_{0} \leq \frac{c}{\|A\|^{D}}(r-r_{+})^{D\tau}, \label{ep0}
	\end{equation}
	then there exist $B(\theta)\in C^{\omega}_{r_{+}}(2\mathbb{T}^{d},{\rm SL}(2,\mathbb{R}))$, $A_{+}\in {\rm SL}(2,\mathbb{R})$ and $f_{+}(\theta) \in C^{\omega}_{r_{+}}(\mathbb{T}^{d}, {\rm sl}(2,\mathbb{R}))$ such that  $(\alpha,Ae^{f(\theta)})$ is conjugated to $(\alpha, A_{+}e^{f_{+}(\theta)})$ by $B(\theta)$, i.e.
	\begin{equation*}
	B(\theta+\alpha)^{-1}Ae^{f(\theta)}B(\theta) = A_{+}e^{f_{+}(\theta)}.
	\end{equation*}
	More precisely, let $N=\frac{2|\ln \varepsilon_{0}|}{r-r_{+}}$ and $\{e^{2\pi i \rho}, e^{-2\pi i \rho}\}$ be the two eigenvalues of $A$, we can distinguish two cases:
	
	{\bf {(A)}}{\rm (Non-resonant case)} Assume that
	\begin{equation*}
	\|2\rho- \langle n,\alpha\rangle\|_{\mathbb{R}/\mathbb{Z}}\geq \varepsilon_{0}^{\sigma},
	\ \ \forall \ n\in\mathbb{Z}^{d} \ \ with \ \ 0<|n|\leq N,
	\end{equation*}
	then we have estimates:
    \begin{equation*}
    |f_{+}(\theta)|_{r_{+}} \leq 4\varepsilon_{0}^{3-2\sigma}, \ |B(\theta)-{\rm Id}|_{r_{+}}\leq \varepsilon_{0}^{\frac{1}{2}}, \ \|A_{+}-A\|\leq 2\|A\|\varepsilon_{0}.
    \end{equation*}	

    {$\bf {(B)}$}{\rm (Resonant case)} If there exists $n^{*}\in \mathbb{Z}^{d}$ with $0<|n^{*}|\leq N$ such that
	\begin{equation*}
	\|2\rho- \langle n^{*},\alpha\rangle\|_{\mathbb{R}/\mathbb{Z}}< \varepsilon_{0}^{\sigma},
	\end{equation*}
	then
	\begin{equation}
	B(\theta) = P \circ e^{Y(\theta)} \circ R_{\frac{\langle n^{*},\theta \rangle}{2}}, \label{com}
	\end{equation}
	 where $P\in {\rm SL}(2,\mathbb{R}), Y\in C^{\omega}_{r_{+}}(\mathbb{T}^{d},{\rm sl}(2,\mathbb{R}))$ and $R_{\frac{\langle n^{*},\theta \rangle}{2}}$ is the rotation matrix as in Section \ref{FRN} with estimates:	
	\begin{equation}\label{B}
	|B(\theta)|_{r_{+}} \leq 4\|A\|^{\frac{1}{2}}\gamma^{-\frac{1}{2}}|n^{*}|^{\frac{\tau}{2}}e^{\pi |n^{*}|r_{+}}, \  |f_{+}(\theta)|_{r_{+}} \ll \varepsilon_{0}^{100}.
	\end{equation}
	Moreover, $\deg B(\theta) = n^{*}$ and the constant $A_{+}$ can be written as
	\begin{equation}
	A_{+} = M^{-1}\exp 2\pi \begin{pmatrix} it_{+} & \nu_{+} \\\overline{\nu_{+}} & - it_{+}\end{pmatrix}M \label{A+}
	\end{equation}
	with estimates $|\nu_{+}|\leq 4\|A\|\gamma^{-1}|n^{*}|^{\tau}\varepsilon_{0} e^{-2\pi |n^{*}|r}$ and $|t_{+}|\leq \frac{3}{5}\varepsilon_{0}^{\sigma}$, where $M=\frac{1}{1+i}\begin{pmatrix} 1 & -i \\1 & i \end{pmatrix}$ and $t_{+}\in \mathbb{R}, \nu_{+}\in \mathbb{C}$. Let  $A_{+}:=e^{A''}$ with $A''\in {\rm sl}(2,\mathbb{R})$ and  $\{e^{2\pi i \rho_{+}}, e^{-2\pi i \rho_{+}}\}$ be the two eigenvalues of $A_{+}$, then
	\begin{equation*}
	\|A''\|\leq 8\varepsilon_{0}^{\sigma}, \ \ |\rho_{+}|\leq 2\varepsilon_{0}^{\sigma}.
	\end{equation*}
\end{proposition}

\begin{remark}
	{\rm Proposition \ref{pr1}} has been proved in \cite{Cai2} essentially, however, the arguments of {\rm (\ref{B})} and {\rm (\ref{A+}) }are new, so we will give a brief proof.  Moreover, we can choose for example $D=1000$.
\end{remark}

\begin{proof}
	It is enough to deal with the resonant case. Combine the resonant condition $\|2\rho-\langle n^{*},\alpha\rangle\|_{\mathbb{R}/\mathbb{Z}}<\varepsilon_{0}^{\sigma}$ with  $\|\langle n^{*},\alpha\rangle\|_{\mathbb{R}/\mathbb{Z}}\geq \frac{\gamma }{|n^{*}|^{\tau}}$, one can get that
	\begin{equation*}
	\frac{\gamma}{|n^{*}|^{\tau}}\leq \|\langle n^{*},\alpha\rangle\|_{\mathbb{R}/\mathbb{Z}} \leq \varepsilon_{0}^{\sigma}+2|\rho|\leq \frac{\gamma}{2|n^{*}|^{\tau}} +2|\rho|,
	\end{equation*}
	hence $|\rho|\geq \frac{\gamma}{4|n^{*}|^{\tau}}$. We only consider the elliptic case since it belongs to the non-resonant case if  $\rho\in i\mathbb{R}$.
	By Lemma 8.1 in \cite{Hou}, one can find constant matrix $P\in {\rm SL}(2,\mathbb{R})$ to diagonalize $A$, i.e.,
	\begin{equation*}
	P^{-1}AP=M^{-1}\exp\begin{pmatrix} {2\pi i \rho} & 0\\ 0 & {-2\pi i \rho} \end{pmatrix}M:=A'\in {\rm SO}(2,\mathbb{R}).
	\end{equation*}
	with
	\begin{equation*}
	\|P\|\leq \sqrt{\frac{2\|A\|}{|\rho|}}\leq 2\sqrt{2}\|A\|^{\frac{1}{2}}\gamma^{-\frac{1}{2}} |n^{*}|^{\frac{\tau}{2}}.
	\end{equation*}
	Denote $g(\theta) = P^{-1}f(\theta)P\in C^{\omega}_{r}(\mathbb{T}^{d},{\rm sl}(2,\mathbb{R}))$, we have
	\begin{equation*}
	|g|_{r}\leq \|P\|^{2} |f|_{r} \leq 8\|A\|\gamma^{-1} |n^{*}|^{\tau}\varepsilon_{0}.
	\end{equation*}
	
	Let $\mathfrak{B}_{r}:=\{f\in C^{\omega}_{r}(\mathbb{T}^{d},{\rm sl}(2,\mathbb{R})), |f|_{r}<\infty \}$.
	For given $\eta>0$, $\alpha\in \mathbb{R}^{d}$ and $A\in {\rm SL}(2,\mathbb{R})$, we introduce a decomposition $\mathfrak{B}_{r}=\mathfrak{B}_{r}^{nre}(\eta)\oplus\mathfrak{B}_{r}^{re}(\eta)$ such that for any $Y \in \mathfrak{B}_{r}^{nre}(\eta)$,
	\begin{equation}
	A^{-1}Y(\theta+\alpha)A\in \mathfrak{B}_{r}^{nre}(\eta),\ \ \ |A^{-1}Y(\theta+\alpha)A-Y(\theta)|_{r}\geq \eta |Y|_{r}. \label{decom}
	\end{equation}
	Now we define
	\begin{equation*}
	\begin{split}
	\Lambda_{1}(\varepsilon_{0}^{\sigma}) &:=\{n\in \mathbb{Z}^{d}: \|\langle n,\alpha\rangle\|_{\mathbb{R}/\mathbb{Z}} \geq \varepsilon_{0}^{\sigma}\}, \\
	\Lambda_{2}(\varepsilon_{0}^{\sigma}) &:=\{n\in \mathbb{Z}^{d}: \|2\rho-\langle n,\alpha\rangle\|_{\mathbb{R}/\mathbb{Z}} \geq \varepsilon_{0}^{\sigma}\},
	\end{split}
	\end{equation*}
	and let $\eta=\varepsilon_{0}^{\sigma}$, one can decompose $\mathfrak{B}_{r}=\mathfrak{B}_{r}^{nre}(\varepsilon_{0}^{\sigma})\oplus\mathfrak{B}_{r}^{re}(\varepsilon_{0}^{\sigma})$ as  in (\ref{decom}) with $A$ substituted by $A'$. A simple calculation yields that $\forall \ Y\in \mathfrak{B}^{nre}_{r}(\varepsilon_{0}^{\sigma})$, $MY(\theta)M^{-1}$ takes the form
	\begin{equation*}
	\sum_{n\in \Lambda_{1}(\varepsilon_{0}^{\sigma})}\begin{pmatrix} i\widehat{t}(n) & 0 \\0 & -i\widehat{t}(n) \end{pmatrix}e^{2\pi i \langle n,\theta \rangle}  + \sum_{n\in \Lambda_{2}(\varepsilon_{0}^{\sigma})} \begin{pmatrix} 0 &\widehat{\nu}(n)e^{2\pi i \langle n,\theta \rangle}\\ \overline{\widehat{\nu}(n)}e^{-2\pi i \langle n,\theta \rangle} &0\end{pmatrix}.
	\end{equation*}	
	In order to eliminate all the non-resonant terms, we recall the following crucial lemma.

	\begin{lemma}[\cite{Cai2}] \label{imp}
		Assume that $A\in {\rm SU}(1,1)$, $\varepsilon_{1}\leq (4\|A\|)^{-4}$, and $\eta \geq 13\|A\|^{2}\varepsilon_{1}^{\frac{1}{2}}$. For any $g\in \mathfrak{B}_{r}$ with $|g|_{r}\leq \varepsilon_{1}$, there exist $Y\in \mathfrak{B}_{r}$ and $g^{re}\in \mathfrak{B}_{r}^{re}(\eta)$ such that
		\begin{equation*}
		e^{Y(\theta+\alpha)}Ae^{g(\theta)}e^{-Y(\theta)}  = Ae^{g^{re}(\theta)},
		\end{equation*}
		with $|Y|_{r}\leq \varepsilon_{1}^{\frac{1}{2}}$, and $|g^{re}|_{r}\leq 2\varepsilon_{1}$.
	\end{lemma}
	\begin{remark}
		Note that {\rm Lemma \ref{imp}} is stated for $\rm SU(1,1)$ and $\rm su(1,1)$, it also works for ${\rm SL}(2,\mathbb{R})$ and ${\rm sl}(2,\mathbb{R})$ since they are isomorphic to $\rm SU(1,1)$ and $\rm su(1,1)$ respectively.
	\end{remark}
	
	According to (\ref{ep0}), one can check that $\varepsilon_{0}^{\sigma}\geq 13\|A'\|^{2} (8\|A\|\gamma^{-1}N^{\tau} \varepsilon_{0})^{\frac{1}{2}}$. By Lemma \ref{imp}, one can find $Y\in \mathfrak{B}_{r}$ and $g^{re}\in \mathfrak{B}_{r}^{re}(\varepsilon_{0}^{\sigma})$ such that
	\begin{equation*}
	e^{-Y(\theta+\alpha)}A'e^{g(\theta)}e^{Y(\theta)} =A'e^{g^{re}(\theta)},
	\end{equation*}
	with estimates
	\begin{equation}\label{normgre}
	|Y|_{r_{+}}\leq 2\sqrt{2}\|A\|^{\frac{1}{2}}\gamma^{-\frac{1}{2}}|n^{*}|^{\frac{\tau}{2}} \varepsilon_{0}^{\frac{1}{2}},\ \ |g^{re}|_{r}\leq 16\|A\|\gamma^{-1} |n^{*}|^{\tau}\varepsilon_{0} \leq \varepsilon_{0}^{1-\frac{\sigma}{5}}.
	\end{equation}
	Recall that in \cite{Cai2}, by $\alpha\in {\rm DC}_{d}(\gamma,\tau)$ we have that $n^{*}$ is the unique resonant site satisfying $0<|n^{*}|\leq N$. Moreover,
	\begin{equation*}
	\begin{split}
	&\Lambda_{1}^{c}(\varepsilon_{0}^{\sigma})\cap \{n\in \mathbb{Z}^{d}:|n|\leq \gamma^{\frac{1}{\tau}}\varepsilon_{0}^{-\frac{\sigma}{\tau}} \} =\{0\}, \\
	&\Lambda_{2}^{c}(\varepsilon_{0}^{\sigma})\cap \{n\in \mathbb{Z}^{d}:|n|\leq 2^{-\frac{1}{\tau}}\gamma^{\frac{1}{\tau}}\varepsilon_{0}^{-\frac{\sigma}{\tau}} \} =\{n^{*}\}.
	\end{split}
	\end{equation*}
    Since $2^{-\frac{1}{\tau}}\gamma^{\frac{1}{\tau}}\varepsilon_{0}^{-\frac{\sigma}{\tau}}-N \gg N$, then $g^{re} \in \mathfrak{B}_{r}^{re}(\varepsilon_{0}^{\sigma})$ can be rewritten as
	\begin{equation*}
	\begin{split}
	g^{re} & = g_{0}^{re} + g_{1}^{re}(\theta) + g_{2}^{re}(\theta)\\
	& = M^{-1}\begin{pmatrix} i\widehat{t}(0) & 0 \\0 & -i\widehat{t}(0) \end{pmatrix} M + M^{-1}\begin{pmatrix} 0 &\widehat{\nu}(n_{*})e^{2\pi i \langle n_{*},\theta \rangle}\\ \overline{\widehat{\nu}(n_{*})}e^{-2\pi i \langle n_{*},\theta \rangle} &0\end{pmatrix} M\\
	& \ \ \ + M^{-1}\sum_{|n|>N^{\prime}} \widehat{g}^{re}(n)e^{2\pi i \langle n,\theta \rangle}M,
	\end{split}
	\end{equation*}
	where $N'=2^{-\frac{1}{\tau}}\gamma^{\frac{1}{\tau}}\varepsilon_{0}^{-\frac{\sigma}{\tau}}-N$.
	In the spirit of Hou-You \cite{Hou}, we perform a conjugation of rotation so that $g_1^{re}(\theta)$ becomes independent of $\theta$. Let $Q(\theta) = R_{\frac{\langle n^{*},\theta \rangle}{2}}$,  we have
	\begin{equation*}
	Q(\theta+\alpha)^{-1}A^{\prime} e^{g^{re}(\theta)}Q(\theta) = \tilde{A}e^{\tilde{g}(\theta)},
	\end{equation*}
	where
	\begin{equation*}
	\tilde{A} = Q(\theta+\alpha)^{-1}A^{\prime}Q(\theta) = M^{-1}\exp \begin{pmatrix} {2\pi i (\rho-\frac{\langle n^{*},\alpha \rangle}{2}}) & 0\\ 0 & -2\pi i (\rho-\frac{\langle n^{*},\alpha \rangle}{2}) \end{pmatrix}M,
	\end{equation*}
	and
	\begin{equation*}
	\begin{split}
	\tilde{g}(\theta) &= Q(\theta)^{-1}g^{re}(\theta)Q(\theta) \\
	&=M^{-1}\begin{pmatrix} i\widehat{t}(0) & 0 \\0 & -i\widehat{t}(0) \end{pmatrix} M + M^{-1}\begin{pmatrix} 0 &\widehat{\nu}(n^{*})\\ \overline{\widehat{\nu}(n^{*})} &0\end{pmatrix} M\\
	& \ \ \ + Q(\theta)^{-1}g_{2}^{re}(\theta) Q(\theta).
	\end{split}
	\end{equation*}
	
	For simplicity, we denote
	\begin{equation}\label{jianji}
	\begin{split}
	S&:=M^{-1}\begin{pmatrix} 2\pi i (\rho-\frac{\langle n^{*},\alpha \rangle}{2}) & 0\\ 0 & -2\pi i (\rho-\frac{\langle n^{*},\alpha \rangle}{2}) \end{pmatrix}M,\\
	L&:=M^{-1}\begin{pmatrix} i\widehat{t}(0) & 0 \\0 & -i\widehat{t}(0) \end{pmatrix} M + M^{-1}\begin{pmatrix} 0 &\widehat{\nu}(n^{*})\\ \overline{\widehat{\nu}(n^{*})} &0\end{pmatrix} M,\\
	F&:=Q(\theta)^{-1}g_{2}^{re}(\theta) Q(\theta).
	\end{split}
	\end{equation}
	Let $B:=P\circ e^{Y}\circ Q\in C^{\omega}_{r_{+}}(2\mathbb{T}^{d},{\rm SL}(2,\mathbb{R}))$, we have the following estimates:
	\begin{align*}
	&|B|_{r_{+}}\leq 4\sqrt{2}\|A\|^{\frac{1}{2}}\gamma^{-\frac{1}{2}}|n^{*}|^{\frac{\tau}{2}}e^{\pi |n^{*}|r_{+}}, \\
	& |F|_{r_{+}}\leq \varepsilon_{0}^{1-\frac{\sigma}{5}} e^{-2\pi N'(r-r_{+})}e^{2\pi N r_{+}} \ll \varepsilon_{0}^{100}.
	\end{align*}
	Also we have $\|S\|\leq 2\pi \varepsilon_{0}^{\sigma}$ and $\|L\|\leq 4\varepsilon_{0}^{1-\frac{\sigma}{5}}$, then
	\begin{equation*}
	\tilde{A}e^{\tilde{g}(\theta)} = e^{S}e^{L+F} =e^{S}(e^{L}+\mathcal{O}(F)) = e^{S}e^{L}\left({\rm Id}+e^{-L}\mathcal{O}(F)\right) :=e^{S}e^{L}e^{f_{+}(\theta)},
	\end{equation*}
	where $f_{+} = \ln ({\rm Id}+e^{-L}\mathcal{O}(F))$ with estimate
	\begin{equation*}
	|f_{+}|_{r_{+}} \leq 2|F|_{r_{+}} \leq \varepsilon_{0}^{100}.
	\end{equation*}
	This finshes the proof for the argument (\ref{B}).

Recall the Baker-Campbell-Hausdorff Formula, i.e.
	\begin{equation*}
	\ln (e^{X}e^{Y}) = X+Y+\frac{1}{2}[X,Y]+\frac{1}{12}([X,[X,Y]]+[Y,[Y,X]]) + \cdots
	\end{equation*}
	where $[X,Y] = XY-YX$ denotes the Lie Bracket and $\cdots$ stands for the higher order terms. Define $A_{+}:=e^{S}e^{L}=e^{A''}$, one can check that
	\begin{equation*}
	A''= S+L+ \frac{1}{2}[S,L]+\frac{1}{12}([S,[S,L]]+[L,[L,S]]) + \cdots
	\end{equation*}
	Moreover, let $\{e^{2\pi i \rho_{+}}, e^{-2\pi i \rho_{+}}\}$ be the two eigenvalues of $A_{+}$, then
	\begin{align*}
	&\|A''\| \leq \|S\|+\|L\|+2\|S\|\cdot \|L\|\leq 8\varepsilon_{0}^{\sigma},\\
	&|\rho_{+}|\leq (2\pi)^{-1}\|A''\| \leq 2\varepsilon_{0}^{\sigma}.
	\end{align*}
	Since $A''\in {\rm sl}(2,\mathbb{R})$, it can be written as $A''=2\pi M^{-1}\begin{pmatrix}
	it_{+} & \nu_{+} \\\overline{\nu_{+}} & - it_{+}
	\end{pmatrix}M$. By Baker-Campbell-Hausdorff Formula and the decay of Fourier coefficients, as well as estimates (\ref{normgre}) and (\ref{jianji}), it follows that
	\begin{align*}
	&|\nu_{+}|\leq  \frac{1}{2\pi}|\widehat{\nu}(n^{*})| + |\widehat{\nu}(n^{*})|\varepsilon_{0}^{\sigma} \leq \frac{1}{4}|\widehat{\nu}(n^{*})| \leq 4\|A\|\gamma^{-1} |n^{*}|^{\tau}\varepsilon_{0} e^{-2\pi |n^{*}|r},\\
	&|t_{+}| \leq \frac{1}{2} \varepsilon_{0}^{\sigma} +\frac{1}{2\pi}|\widehat{t}(0)|+ \frac{1}{2} |\widehat{\nu}(n^{*})|\varepsilon_{0}^{\sigma} \leq \frac{3}{5} \varepsilon_{0}^{\sigma}.
	\end{align*}
This finishes the proof.
\end{proof}

With Proposition \ref{pr1} in hand, we can apply it inductively to the approximating sequence of analytic cocycles and bring the estimates back to $C^k$ cocycles by analytic approximation. We shall formulate the $C^k$ almost reducibility in the following subsection.

\subsection{Differentiable quantitative almost reducibility}
Let $\{f_{j}\}_{j\geq 1}, f_{j}\in C^{\omega}_{\frac{1}{j}}(\mathbb{T}^{d},{\rm sl}(2,\mathbb{R}))$ be the analytic sequence approximating $f\in C^{k}(\mathbb{T}^{d},{\rm sl}(2,\mathbb{R}))$.  We first recall some notations given in \cite{Cai2}. Let $c=c(\gamma,\tau,d)$ and $D$ be the constants defined by Proposition \ref{pr1}, and we denote
\[
\varepsilon_{0}^{\prime}(h,h^{\prime}):=\frac{c}{(2\|A\|)^{D}}(h-h^{\prime})^{D\tau},
\]
and define
\begin{equation*}
\varepsilon_{m}:=\frac{c}{(2\|A\|)^{D}m^{\frac{k}{4}}}.
\end{equation*}
Then one can check that for any $k\geq 5D\tau$  and any $m\geq 10, m\in \mathbb{Z}$,
\begin{equation*}
\frac{c}{(2\|A\|)^{D}m^{\frac{k}{4}}} \leq \varepsilon_{0}^{\prime}(\frac{1}{m},\frac{1}{m^{2}}).
\end{equation*}
Denote $l_{j}=M^{2^{j-1}}$, $\forall \ j\in \mathbb{Z}^{+}$, where $M>\max\{10,\frac{(2\|A\|)^{D}}{c}\}$ is an integer.

\begin{theorem}[\cite{Cai2}]	\label{ar}
	Let $\alpha\in {\rm DC}_{d}(\gamma,\tau)$, $A\in {\rm SL}(2,\mathbb{R})$, $\sigma=\frac{1}{10}$, $f(\theta)\in C^{k}(\mathbb{T}^{d},{\rm sl}(2,\mathbb{R}))$ with $k\geq 5D\tau$. Let $\{f_{j}\}_{j\geq 1}$ be the analytic sequence approximating $f(\theta)$ defined in $(\ref{aa})$. There exists $\varepsilon_{2} = \varepsilon_{2}(\gamma, \tau, d,k,\|A\|)$ such that if $\|f\|_{k}\leq \varepsilon_{2}$, then the following holds:
	
	{\bf {(A)}}
	There exist $B_{l_{j}}(\theta)\in C^{\omega}_{\frac{1}{l_{j+1}}}(2\mathbb{T}^{d}, {\rm SL}(2,\mathbb{R}))$, $A_{l_{j}}\in {\rm SL}(2,\mathbb{R})$ and $f^{\prime}_{l_{j}}(\theta)\in C^{\omega}_{\frac{1}{l_{j+1}}}(\mathbb{T}^{d}, {\rm sl}(2,\mathbb{R}))$ such that
	\begin{equation*}
	B_{l_{j}}(\theta+\alpha)^{-1}Ae^{f_{l_{j}}(\theta)}B_{l_{j}}(\theta) = A_{l_{j}}e^{f_{l_{j}}^{\prime}(\theta)}, \label{ao1}
	\end{equation*}
	with following estimates
	\begin{equation*}
	|B_{l_{j}}(\theta)|_{\frac{1}{l_{j+1}}} \leq \varepsilon_{l_{j}}^{-\frac{2\sigma}{5}}, \
	|f_{l_{j}}^{\prime}(\theta)|_{\frac{1}{l_{j+1}}} \leq \frac{1}{2}\varepsilon_{l_{j}}^{\frac{5}{2}} , \  \|A_{l_{j}}\|\leq 2\|A\|.
	\end{equation*}

	{\bf {(B)}}
	There exists $\tilde{f}_{l_{j}}(\theta)\in C^{k_{0}}(\mathbb{T}^{d},{\rm sl}(2,\mathbb{R}))$ with $k_{0}\leq [\frac{k}{20}]$ such that
	\begin{equation*}
	B_{l_{j}}(\theta+\alpha)^{-1}Ae^{f(\theta)}B_{l_{j}}(\theta) = A_{l_{j}}e^{\tilde{f}_{l_{j}}(\theta)}, \label{ao3}
	\end{equation*}
	with estimate
	\begin{equation*}
	\|\tilde{f}_{l_{j}}(\theta)\|_{k_{0}}\leq \varepsilon_{l_{j}}^{2}. \label{ao4}
	\end{equation*}

	{\bf {(C)}}
Let $\{e^{2\pi i \rho_{j-1}}, e^{-2\pi i \rho_{j-1}}\}$ be the two eigenvalues of $A_{l_{j-1}}$ $(A_{l_{0}} =A \ {\rm if} \ j=1)$. If the $({j})$-th step is obtained by resonant case, i.e.
there exists $n^{*}_{l_{j}}\in \mathbb{Z}^{d}$ such that
\begin{equation}
\|2\rho_{j-1}-\langle n^{*}_{l_{j}},\alpha \rangle\|_{\mathbb{R}/\mathbb{Z}}<\varepsilon_{l_{j}}^{\sigma}, \ \ \ 0<|n^{*}_{l_{j}}|\leq N_{l_{j}}:=\frac{2|\ln \varepsilon_{l_{j}}|}{\frac{1}{l_{j}}-\frac{1}{l_{j+1}}}, \label{gz}
\end{equation}
then $A_{l_{j}}$ has the following form
\begin{equation}
A_{l_{j}} = M^{-1}\exp 2\pi \begin{pmatrix} it_{l_{j}} & \nu_{l_{j}} \\\overline{\nu_{l_{j}}} & - it_{l_{j}}\end{pmatrix}M  \label{ao5}
\end{equation}	
with $|\nu_{l_{j}}|\leq 4\|A\|\gamma^{-1}|n^{*}_{l_{j}}|^{\tau}\varepsilon_{l_{j}} e^{-2\pi \frac{1}{l_{j}}|n^{*}_{l_{j}}|}$. Let  $\{e^{2\pi i \rho_{j}}, e^{-2\pi i \rho_{j}}\}$ be the two eigenvalues of $A_{l_{j}}$, then $|\rho_{j}|\leq 2 \varepsilon^{\sigma}_{l_{j}}$.
\end{theorem}
\begin{remark}
	The arguments {\bf {(A)}} and {\bf {(B)}} have been proved in \cite{Cai} and \cite{Cai2} essentially. However the estimates in {\bf {(C)}} are new, so we will give a brief proof for them.
\end{remark}
\begin{proof}
In order to prove ({\bf C}), we take a quick glimpse at the proof process of ({\bf A}).	Suppose that the argument {\bf {(A)}}  holds for $({j-1})$-th step, i.e.,
	\begin{equation*}
	B_{l_{j-1}}(\theta+\alpha)^{-1}Ae^{f_{l_{j-1}}(\theta)}B_{l_{j-1}}(\theta) = A_{l_{j-1}}e^{f_{l_{j-1}}^{\prime}(\theta)},
	\end{equation*}
	with following estimates
	\begin{equation} \label{sb}
	|B_{l_{j-1}}(\theta)|_{\frac{1}{l_{j}}} \leq \varepsilon_{l_{j-1}}^{-\frac{2\sigma}{5}}, \
	|f_{l_{j-1}}^{\prime}(\theta)|_{\frac{1}{l_{j}}} \leq \frac{1}{2}\varepsilon_{l_{j-1}}^{\frac{5}{2}} , \  \|A_{l_{j-1}}\|\leq 2\|A\|.
	\end{equation}
	Then for the cocycle $(\alpha,Ae^{f_{l_{j}}(\theta)})$, one can directly calculate that
	\begin{equation*}
	B_{l_{j-1}}(\theta+\alpha)^{-1}Ae^{f_{l_{j}}}B_{l_{j-1}} =A_{l_{j-1}}e^{f'_{l_{j-1}}} +B_{l_{j-1}}(\theta+\alpha)^{-1}(Ae^{f_{l_{j}}}-Ae^{f_{l_{j-1}}})B_{l_{j-1}}.
	\end{equation*}
	If we rewrite that
	\begin{equation*}
	A_{l_{j-1}}e^{\bar{f}_{l_{j-1}}(\theta)}:=A_{l_{j-1}}e^{f'_{l_{j-1}}} +B_{l_{j-1}}(\theta+\alpha)^{-1}(Ae^{f_{l_{j}}}-Ae^{f_{l_{j-1}}})B_{l_{j-1}}(\theta),
	\end{equation*}
then by (\ref{sb}), we have
	\begin{equation*}
\begin{split}
|\bar{f}_{l_{j-1}}(\theta)|_{\frac{1}{l_{j}}} &\leq |f_{l_{j-1}}'|_{\frac{1}{l_{j}}}+ \|A_{l_{j-1}}^{-1}\|\cdot|B_{l_{j-1}}(\theta+\alpha)^{-1}(Ae^{f_{l_{j}}}-Ae^{f_{l_{j-1}}})B_{l_{j-1}}|_{\frac{1}{l_{j}}}\\
&\leq \frac{1}{2}\varepsilon_{l_{j-1}}^{\frac{5}{2}}+2\|A\|^{2}\times
\varepsilon_{l_{j-1}}^{-\frac{4\sigma}{5}}\times \frac{c}{(2\|A\|)^{D}l_{j-1}^{k-1}}\\
&\leq \frac{1}{2}\varepsilon_{l_{j}} + \frac{1}{2}\times \frac{c}{(2\|A\|)^{D}l_{j}^{\frac{k}{4}}}\\
&\leq \varepsilon_{l_{j}}.
\end{split}
\end{equation*}

Let us focus on the cocycle $(\alpha, A_{l_{j-1}}e^{\bar{f}_{l_{j-1}}(\theta)})$. Since
$({j})$-th step is obtained by resonant case,
apply Proposition \ref{pr1} and resonant condition (\ref{gz}), $A_{l_{j}}$ in the $({j})$-th step can be written as
\begin{equation*}
A_{l_{j}} = M^{-1}\exp 2\pi \begin{pmatrix}  it_{l_{j}}& \nu_{l_{j}} \\\overline{\nu_{l_{j}}} & - it_{l_{j}}\end{pmatrix}M,
\end{equation*}
with  estimates:
\begin{equation}
|\nu_{l_{j}}|\leq 4\|A\|\gamma^{-1}|n^{*}_{l_{j}}|^{\tau}\varepsilon_{l_{j}} e^{-2\pi \frac{1}{l_{j}}|n^{*}_{l_{j}}|}, \ \ \ |\rho_{j}|\leq 2 \varepsilon^{\sigma}_{l_{j}}.
\end{equation}	
This gives the estimates of ({\bf C}).
\end{proof}
\subsection{Reducibility of $C^{k}$ quasi-periodic cocycle}
With an extra assumption on the rotation number of the initial system, it is possible to prove that almost reducibility leads to reducibility in the sense that the number of resonances is finite. Actually, the quantitative estimates are closely related to the condition of the rotation number, as will be shown in the following.
\begin{theorem}\label{r}
	Let $\alpha\in {\rm DC}_{d}(\gamma,\tau)$, $\sigma = \frac{1}{10}$, $f(\theta)\in C^{k}(\mathbb{T}^{d}, {\rm sl}(2,\mathbb{R}))$ with $k\geq 5D\tau$ and $A\in {\rm SL}(2,\mathbb{R})$. Assume that the cocycle $(\alpha, Ae^{f(\theta)})$ is not uniformly hyperbolic and $2\rho(\alpha, Ae^{f(\theta)}) = \langle m,\alpha \rangle \mod \mathbb{Z}$ for $m\in \mathbb{Z}^{d}\backslash \{0\}$. There exists $\varepsilon_{3}=\varepsilon_{3}(\gamma, \tau,d,k, \|A\|)$ such that if $\|f\|_k\leq \varepsilon_{3}$, then there exists
	$B\in C^{\tilde{k}}(2\mathbb{T}^{d},{\rm SL}(2,\mathbb{R}))$ such that
	\begin{equation*}
	B(\theta+\alpha)^{-1}Ae^{f(\theta)}B(\theta) = \begin{pmatrix} 1&\zeta\\0&1
	\end{pmatrix},
	\end{equation*}
	with estimate  $|\zeta| \leq \varepsilon_{3}^{\frac{1}{3}}|m|^{-\frac{k}{8.5}}$ and $\|B(\theta)\|_{\tilde{k}}\leq D_{1} |m|^{(2\tilde{k}+\tau)}$ with $ \tilde{k}\leq[\frac{k}{400}]$, $D_{1} = D_{1}(\gamma,\tau,d,k,\|A\|)$.
\end{theorem}

\begin{remark}
	We are not going to deal with the uniformly hyperbolic cocycle as it is always reducible in our context and has nothing to do with our spectral application \emph(in Schr\"{o}dinger case, it corresponds to energy $E$ lying in the gap\emph). Moreover, $\zeta$ is the key to derive gap estimates.
\end{remark}

\begin{proof}
	 The result can be proved by Theorem \ref{ar} iteratively. Take $\varepsilon_{3}=\varepsilon_{2}$ as in Theorem \ref{ar} and apply it to cocycle $(\alpha, Ae^{f(\theta)})$, then there exist  $B_{l_{j}}(\theta)\in C^{\omega}_{\frac{1}{l_{j+1}}}(2\mathbb{T}^{d},{\rm SL}(2,\mathbb{R}))$, $A_{l_{j}}\in
	 {\rm SL}(2,\mathbb{R})$ and $\tilde{f}_{l_{j}}(\theta)\in C^{k_{0}}(\mathbb{T}^{d},{\rm sl}(2,\mathbb{R}))$ with $k_{0} \leq [\frac{k}{20}]$ such that
	 \begin{equation*}
	 B_{l_{j}}(\theta+\alpha)^{-1}Ae^{f(\theta)}B_{l_{j}}(\theta) = A_{l_{j}}e^{\tilde{f}_{l_{j}}(\theta)}
	 \end{equation*}
	 with $\|\tilde{f}_{l_{j}}(\theta)\|_{k_{0}}\leq \varepsilon_{l_{j}}^{2}$.
	
	 Since $\rho(\alpha, Ae^{f(\theta)})= \frac{\langle m,\alpha \rangle}{2} \mod \frac{\mathbb{Z}}{2}, m\in \mathbb{Z}^{d}\backslash\{0\}$, by (\ref{rho}) we have
	 \begin{equation*}
	 \begin{split}
	 \rho(\alpha, A_{l_{j}}e^{\tilde{f}_{l_{j}}(\theta)}) &= \rho(\alpha, Ae^{f(\theta)})-\frac{\langle \deg B_{l_{j}}, \alpha \rangle}{2} \mod \frac{\mathbb{Z}}{2}\\
	 &=\frac{\langle m-\deg B_{l_{j}}, \alpha \rangle}{2} \mod \frac{\mathbb{Z}}{2}.
	 \end{split}
	 \end{equation*}
	 From now on, we omit ``${\rm mod} \ \frac{\mathbb{Z}}{2}$" for simplicity and recall the following important lemma.
	 \begin{lemma}[\cite{Cai}]
	 	Let $\alpha \in {\rm DC}_{d}(\gamma,\tau)$, $f(\theta) \in C^{k}(\mathbb{T}^{d},{\rm sl}(2,\mathbb{R}))$ with $k\geq 5D\tau$ and $A\in {\rm SL}(2,\mathbb{R})$. Assume that $\rho(\alpha,Ae^{f(\theta)}) = 0$, there exists $T=T(\tau)$ and $\varepsilon_{4}=\varepsilon_{4}(\gamma,\tau,d,k,\|A\|)$ such that if
	 	\begin{equation*}
	 	\|f\|_{k}\leq \varepsilon_{4}\leq T(\tau)\gamma^{11}\varepsilon_{0}'(\frac{1}{l_{1}},\frac{1}{l_{2}}),
	 	\end{equation*}
	 	then there exist $B_{1}\in C^{\tilde{k}}(\mathbb{T}^{d},{\rm SL}(2,\mathbb{R}))$ with $\tilde{k}=[\frac{k}{20}]$ and $H\in {\rm SL}(2,\mathbb{R})$, $\tilde{\delta}=\tilde{\delta}(\gamma,\tau,d,k,\|A\|)>0$ such that
	 	\begin{equation*}
	 	B_{1}(\theta+\alpha)^{-1}A^{f(\theta)}B_{1}(\theta) = H
	 	\end{equation*}
	 	with estimates $\|B_{1}-{\rm Id}\|_{\tilde{k}} \leq \tilde{\delta}$, $\deg B_{1}(\theta)=0$ and $\|A-H\|\leq 4\|A\|\varepsilon_{4}.$
	 	\label{sim}
	 \end{lemma}
	
	 Let us analyze the structure of resonances firstly and
	 assume that there exist at least two resonant steps,  say the $({j_{i}})$-th and $({j_{i+1}})$-th, in the almost reducibility procedure.
	 On one hand,
	 for the $({j_{i+1}})$-th step, by using the resonant condition $\| 2\rho_{{j_{i+1}-1}}-\langle n_{l_{j_{i+1}}}^{*}, \alpha \rangle\|_{\mathbb{R}/\mathbb{Z}} \leq \varepsilon_{l_{j_{i+1}}}^{\sigma}$ and the Diophantine condition $\|\langle n^{*}_{l_{j_{i+1}}},\alpha \rangle \|_{\mathbb{R}/\mathbb{Z}}\geq \gamma |n^{*}_{l_{j_{i+1}}}|^{-\tau}$, one can get $|\rho_{{j_{i+1}-1}}|\geq \frac{1}{3}\gamma |n_{l_{j_{i+1}}}^{*}|^{-\tau }$.
	 On the other hand, by Theorem \ref{ar},   we have $|\rho_{{j_{i}}}|\leq  2\varepsilon_{l_{j_{i}}}^{\sigma}$ after the $({j_{i}})$-th step.
	 Then $|\rho_{{j_{i+1}-1}}|\leq 4 \varepsilon_{l_{j_{i}}}^{\sigma} \leq  \frac{1}{3}\gamma |n_{l_{j_{i}}}^{*}|^{-2\tau } \varepsilon_{l_{j_{i}}}^{\frac{1}{25}}$. Thus
	 \begin{equation}
	 |n_{l_{j_{i+1}}}^{*}| \geq \varepsilon_{l_{j_{i}}}^{-\frac{1}{25\tau}} |n_{l_{j_{i}}}^{*}|^{2}.
	 \label{bili}
	 \end{equation}
	
	 Recall that $\deg B_{l_{j+1}} = \deg B_{l_{j}} + n_{l_{j+1}}^{*}$ if $({j+1})$-th is obtained by resonant case. By (\ref{bili}), it follows that there are at most finitely many resonant steps before $\rho(\alpha, A_{l_{j}}e^{\tilde{f}_{l_{j}}(\theta)})=0$. Once we get $\rho(\alpha, A_{l_{j}}e^{\tilde{f}_{l_{j}}(\theta)})=0$ in $({j})$-th step, the conjugacies that after $({j})$-th step will be close to the identity (i.e. non-resonant steps, and the rotation number remains zero by (\ref{rho})). Hence by Theorem \ref{ar} one can choose the smallest $j'>j$ such that
	 \begin{equation*}
	 \|\tilde{f}_{l_{j'}}(\theta)\|_{k_{0}}\leq \varepsilon_{l_{j'}}^{2}\leq T(\tau)\gamma^{11}\varepsilon_{0}'(\frac{1}{l_{1}},\frac{1}{l_{2}}).
	 \end{equation*}
	 Applying Lemma \ref{sim} to the cocycle $(\alpha, A_{l_{j'}}e^{\tilde{f}_{l_{j'}}(\theta)})$,  there exist $H\in {\rm SL}(2,\mathbb{R})$ and $B_{1}\in C^{\tilde{k}}(\mathbb{T}^{d},{\rm SL}(2,\mathbb{R}))$ with $\tilde{k}=[\frac{k_{0}}{20}]$ such that
	 \begin{equation*}
	 B_{1}(\theta+\alpha)^{-1}A_{l_{j'}}e^{\tilde{f}_{l_{j'}}(\theta)}B_{1}(\theta) = H
	 \end{equation*}
	 with $\|A_{l_{j'}}-H\|\leq 4\|A\|\varepsilon_{l_{j'}}^{2}$ and $\|B_{1}\|_{\tilde{k}} \leq 2$ with $\deg B_{1}(\theta)=0$.

	 We conclude that there exist finitely many resonant steps in the reducibility procedure when  $\rho(\alpha, Ae^{f(\theta)})= \frac{\langle m,\alpha \rangle}{2}, m\in \mathbb{Z}^{d}\backslash\{0\}$.
	 Since $\deg B_{1}=0$, by (\ref{rho}), we have  $\rho(\alpha, H)=0$.
	 Combine that the cocycle $(\alpha,Ae^{f(\theta)})$ is not uniformly hyperbolic, we deduce that $H$ is parabolic. As $H\in {\rm SL}(2,\mathbb{R})$, we have $H=e^{h}$ with $h\in {\rm sl}(2,\mathbb{R})$ and $\det h=0$. Assume that $h=\begin{pmatrix}
	 h_{11}&h_{12}\\
	 h_{21}&-h_{11}
	 \end{pmatrix}$, then there exists $\phi\in \mathbb{T}^{1}$ such that $R_{-\phi}hR_{\phi}= \begin{pmatrix}
	 0&h_{21}-h_{12}\\
	 0&0
	 \end{pmatrix}$.
	 Let $B_{2}(\theta)=B_{1}(\theta) \circ R_{\phi}$ and $\zeta= h_{21}-h_{12}$,
	 we can see that the cocycle $(\alpha, Ae^{f(\theta)})$  is conjugated to $\tilde{H}=\begin{pmatrix} 1 & \zeta \\ 0& 1 \end{pmatrix}$ by $B(\theta)$, where $B(\theta) =  B_{l_{j'}}(\theta)\circ B_{2}(\theta)  \in C^{\tilde{k}}(2\mathbb{T}^{d},{\rm SL}(2,\mathbb{R}))$ with $\tilde{k}\leq[\frac{k}{400}]$.

	 Assuming that there are $s+1$ resonant steps with resonant sites:
	 \begin{equation*}
	 n_{l_{j_{0}}}^{*},\cdots, n_{l_{j_{s}}}^{*} \in \mathbb{Z}^{d}, \ \ \ 0<|n_{l_{j_{i}}}^{*}|\leq N_{l_{j_{i}}}=\frac{2|\ln \varepsilon_{l_{j_{i}}}|}{\frac{1}{l_{j_{i}}}-\frac{1}{l_{j_{i}+1}}}, \ \ i=0,1,\cdots, s,
	 \end{equation*}
	 then $m=n_{l_{j_{0}}}^{*}+\cdots +n_{l_{j_{s}}}^{*}$.	In  view of the inequality (\ref{bili}) and the fact that
	 \begin{equation*}
	 |n_{l_{j_{s}}}^{*}|-\sum_{i=0}^{s-1}|n_{l_{{j_{i}}}}^{*}| \leq |m|\leq |n_{l_{j_{s}}}^{*}| + \sum_{i=0}^{s-1}|n_{l_{{j_{i}}}}^{*}|,
	 \end{equation*}
	 we get $(1-2\varepsilon_{3}^{\frac{1}{25\tau }})|n_{l_{j_{s}}}^{*}|\leq |m|\leq (1+2\varepsilon_{3}^{\frac{1}{25\tau }})|n_{l_{j_{s}}}^{*}|$.

	 Let us focus on cocycle $(\alpha, A_{l_{j_{s}}}e^{\tilde{f}_{l_{j_{s}}}(\theta)})$, which means it is obtained by the last resonant step.  In view of (\ref{ao5}),
	 \begin{equation*}
	 A_{l_{j_{s}}} = M^{-1}\exp 2\pi \begin{pmatrix} it_{l_{j_{s}}} & \nu_{l_{j_{s}}} \\\overline{\nu_{l_{j_{s}}}} & - it_{l_{j_{s}}}\end{pmatrix}M :=e^{A''_{l_{j_{s}}}}
	 \end{equation*}
	 with
	 \begin{equation}
	 |\nu_{l_{j_{s}}}|\leq 4\|A\|\gamma^{-1}|n^{*}_{l_{j_{s}}}|^{\tau}\varepsilon_{l_{j_{s}}} e^{-2\pi \frac{1}{l_{j_{s}}} |n^{*}_{l_{j_{s}}}|}. \label{ess}
	 \end{equation}
	 In the following, we will estimate the constant matrix $H$.
	 Rewrite $H$ as $H= M^{-1}\exp \begin{pmatrix} i\beta_{11} & \beta_{12} \\ \bar{\beta}_{12} & -i\beta_{11} \end{pmatrix}M$, where $\beta_{11}\in \mathbb{R}$, $\beta_{12}\in \mathbb{C}$.
	 Since $A''_{l_{j_{s}}} = 2\pi M^{-1}\begin{pmatrix}
	 it_{l_{j_{s}}} & \nu_{l_{j_{s}}} \\\overline{\nu_{l_{j_{s}}}} & - it_{l_{j_{s}}}
	 \end{pmatrix}M$,
	 by  (\ref{ess}) and Lemma \ref{sim}, it follows that
	 \begin{equation}
	 \begin{split}
	 |\beta_{12}|
	 &=|(M(h-A''_{l_{j_{s}}})M^{-1})_{12} + (MA''_{l_{j_{s}}}M^{-1})_{12}|\\
	 &\leq 4\|H-A_{l_{j_{s}}}\| + 2\pi |\nu_{l_{j_{s}}}|\\
	 &\leq 16\|A\|\varepsilon_{l_{j_{s}}}^{2} +  2\pi 4\|A\|\gamma^{-1}|n^{*}_{l_{j_{s}}}|^{\tau} \varepsilon_{l_{j_{s}}} e^{-2\pi \frac{1}{l_{j_{s}}}|n_{l_{j_{s}}}^{*}|} \\
	 &\leq 16\pi\|A\|\gamma^{-1}\times \frac{c}{(2\|A\|)^{D}}  |n^{*}_{l_{j_{s}}}|^{\tau} \left( \frac{1}{l_{j_{s}}}\right) ^{\frac{k}{4}}.
	 \end{split}\label{beta1}
	 \end{equation}
	 The second step uses that fact that $\|X-Y\|\leq 2\|e^{X}-e^{Y}\|$ if $X,Y\in {\rm sl}(2,\mathbb{R})$ and $\|X\|$ and $\|Y\|$ are small enough.
	 By the definition $N_{l_{j}} = \frac{2|\ln \varepsilon_{l_{j}}|}{\frac{1}{l_{j}}-\frac{1}{l_{j+1}}}$ and  $0< |n_{l_{j}}^{*}| \leq N_{l_{j}}$, we deduce that
	 \begin{equation}
	 |n_{l_{j_{s}}}^{*}| \leq \frac{2|\ln \varepsilon_{l_{j_{s}}}|}{\frac{1}{2}\frac{1}{l_{j_{s}}}}
	 =4 l_{j_{s}} \ln \frac{1}{\varepsilon_{l_{j_{s}}}}
	 =4l_{j_{s}} (c^{*}+ \frac{k}{4} \ln l_{j_{s}}),\label{n1}
	 \end{equation}
	 where $c^{*} = D\ln 2\|A\| +\ln \frac{1}{c}$ is a constant. To compare the relation between $|\beta_{12}|$ and $|n_{l_{j_{s}}}^{*}|$, we need the following essential observation. For fixd $k$, let $\xi=10^{-5}$, one can always choose the constant $l_{1}=M$ sufficiently large (and then $l_{j_{s}}$ also sufficiently large by $l_{j_{s}} =  M ^{2^{j_{s}-1}} $) such that
	 \begin{equation} \label{g}
	 (l_{j_{s}})^{1+\xi} \geq 4l_{j_{s}}(c^{*} + \frac{k}{2} \ln l_{j_{s}}),
	 \end{equation}
	 and
	 \begin{equation}
	 \frac{16\pi \|A\|}{\gamma} \times \left(\frac{1}{l_{j_{s}}}\right) ^{\frac{k}{8}+\xi} \leq \frac{1}{4} \left(\frac{1}{l_{1}}\right)^{\frac{k}{8}}.  \label{e1}
	 \end{equation}
	 Then by (\ref{n1}) and (\ref{g}), we have
	 \begin{equation}
	 \begin{split}
	 (l_{j_{s}})^{\frac{k}{8}-\xi} = \left((l_{j_{s}})^{1+\xi}\right) ^{\frac{\frac{k}{8}-\xi}{1+\xi}}
	 \geq |n_{l_{j_{s}}}^{*}|^{\frac{k-8\xi}{8(1+\xi)}}.
	 \end{split}\label{l1}
	 \end{equation}
	 According to (\ref{beta1}), (\ref{e1}) and (\ref{l1}), it follows that
	 \begin{equation*}
	 |\beta_{12}| \leq \frac{c}{(2\|A\|)^{D}}\times \frac{1}{4} \left(\frac{1}{l_{1}} \right)^{\frac{k}{8}} \times |n^{*}_{l_{j_{s}}}|^{\tau} \left( \frac{1}{l_{j_{s}}} \right)^{\frac{k}{8}-\xi} \leq \frac{1}{4}\varepsilon_{l_{1}}^{\frac{1}{2}} |n^{*}_{l_{j_{s}}}|^{-(\frac{k-8\xi}{8(1+\xi)}-\tau)}.
	 \end{equation*}
	 Since $\det h=0$, we have $|\beta_{11}|\leq \frac{1}{4} \varepsilon_{l_{1}}^{\frac{1}{2}} |n^{*}_{l_{j_{s}}}|^{-(\frac{k-8\xi}{8(1+\xi)}-\tau)}$, then
	 \begin{equation*}
	 |h_{12}|, |h_{21}|\leq \frac{1}{2} \varepsilon_{l_{1}}^{\frac{1}{2}} |n^{*}_{l_{j_{s}}}|^{-(\frac{k-8\xi}{8(1+\xi)}-\tau)}.
	 \end{equation*}
	 By the relation $|m|\leq (1+2\varepsilon_{4}^{\frac{1}{25\tau}})|n_{l_{j_{s}}}^{*}|$, we have
	 \begin{equation*}
	 \begin{split}
	 |\zeta| = |h_{12}-h_{21}|\leq \varepsilon_{l_{1}}^{\frac{1}{2}} |n^{*}_{l_{j_{s}}}|^{-(\frac{k-8\xi}{8(1+\xi)}-\tau)}
	 \leq \varepsilon_{3}^{\frac{1}{3}}|m|^{-\frac{k}{8.5}}.
	 \end{split}
	 \end{equation*}
	 By (\ref{com}) in Proposition \ref{pr1} and {\bf (A)} of Theorem \ref{ar} with  the fact $\sum_{i=0}^{s}|n_{l_{j_{i}}}^{*}| \leq \frac{3}{2} |m|$ and $\prod_{i=0}^{s}|n_{l_{j_{i}}}^{*}|\leq |n_{l_{j_{s}}}^{*}|^{2}$, we have
	 \begin{equation*}
	 \|B(\theta)\|_{\tilde{k}}\leq 2\prod_{i=0}^{s} \|P_{l_{j_{i}}}\|  \|e^{Y_{l_{j_{i}}}}\|_{\tilde{k}}  \|Q_{l_{j_{i}}}(\theta)\|_{\tilde{k}} \leq \frac{4|n_{l_{j_{s}}}^{*}|^{2\tilde{k}+\tau}}{\sqrt{\gamma^{s+1}}}  \prod_{i=0}^{s}  \|A_{l_{j_{i}}}\|^{\frac{1}{2}}\|e^{Y_{l_{j_{i}}}}\|_{\tilde{k}}  .
	 \end{equation*}
	 By Cauchy estimate and (\ref{normgre}), one can obtain that
	 \begin{equation*}
	 \begin{split}
	 \prod_{i=0}^{s}\|e^{Y_{l_{j_{i}}}}\|_{\tilde{k}} &\leq \exp \sum_{i=0}^{s} \|Y_{l_{j_{i}}}\|_{\tilde{k}} \leq \exp \sum_{i=0}^{s} \tilde{k}! (l_{j_{i}})^{\tilde{k}} |Y_{l_{j_{i}}}|_{\frac{1}{l_{j_{i}}}}\\
	 & \leq \exp\sum_{i=0}^{s} \left( \tilde{k}!  (l_{j_{i}})^{\tilde{k}}2\sqrt{2} \|A_{l_{j_{i}-1}}\|^{\frac{1}{2}} \gamma^{-\frac{1}{2}} |n^{*}_{l_{j_{i}}}|^{\frac{\tau}{2}} \varepsilon_{l_{j_{i}}}^{\frac{1}{2}}\right)\\
	 &\leq  \exp\sum_{i=0}^{s}  (l_{j_{i}})^{-\frac{k}{20}}<2,
	 \end{split}
	 \end{equation*}
hence we conclude the estimate of $\|B(\theta)\|_{\tilde{k}}$ as	
	 \begin{equation*}
	  \|B(\theta)\|_{\tilde{k}}\leq D_{1} |m|^{(2\tilde{k}+\tau)}
	 \end{equation*}
for some $D_{1} = D_{1}(\gamma, \tau, d,k, \|A\|)$.
\end{proof}

\section{Gap estimates via Moser-P\"{o}schel argument}
For the sake of intriguing application, we specialize in one typical example of quasi-periodic ${\rm SL}(2,\mathbb{R})$ cocycles.
Let us consider the discrete quasi-periodic Schr\"{o}dinger operator on $\ell^{2}(\mathbb{Z})$:
\begin{equation*}
(H_{V,\alpha,\theta}u)_{n}= u_{n+1} + u_{n-1} + V(\theta+n\alpha)u_{n}, \ \ \forall \ n\in \mathbb{Z},
\end{equation*}
where $\alpha\in {\rm DC}_{d}(\gamma,\tau)$, $\theta\in \mathbb{T}^{d}$, and $ V\in C^{k}(\mathbb{T}^{d},\mathbb{R})$ is small. Recall the Gap-Labelling Theorem, each spectral gap has a unique $m\in\mathbb{Z}^{d}$ satisfying $2\rho(\alpha,S_{E}^{V})=\langle m,\alpha \rangle$ and we denote by $G_{m}(V)=(E_{m}^{-},E_{m}^{+})$.
Since $E^{+}_{m}\in {\Sigma_{V,\alpha}}$ is the right edge point of gap $G_{m}(V)$, from Theorem \ref{r} the Schr\"{o}dinger cocycle $(\alpha, S^{V}_{E_{m}^{+}})$ can be $C^{k,\tilde{k}}$  conjugated to a constant parabolic cocycle $(\alpha, B)$ if $\|V\|_{k}$ is sufficient small, i.e., there exist $X\in C^{\tilde{k}}(2\mathbb{T}^{d},{\rm SL}(2,\mathbb{R}))$ with $\tilde{k}\leq [\frac{k}{400}]$, $B\in {\rm SL}(2,\mathbb{R})$ and $\zeta>0$ such that
\begin{equation*}
X(\theta+\alpha)^{-1}S^{V}_{E_{m}^{+}}X(\theta) = B :=\begin{pmatrix}
1&\zeta\\
0&1
\end{pmatrix}.
\end{equation*}
\begin{remark}
When $\zeta<0$, the energy of  Schr\"{o}dinger cocycle lies at left edge point of a gap.	However  $\zeta=0$ if and only if the corresponding energy is in a collapsed spectral gap.
\end{remark}

In this section, we will show that $|G_{m}(V)|$ is determined by $\|X\|_{\tilde{k}}$ and $\zeta$. To achieve this, we first establish our $C^k$ version of Moser-P\"{o}schel argument. We denote $[\cdot]$ the average of a quasi-periodic function.

\subsection{Moser-P\"{o}schel argument}

Assume that $\zeta\in(0,\frac{1}{2})$.
For any $\delta\in (0,1)$, by direct calculation we have
\begin{equation*}
X(\theta+\alpha)^{-1}S^{V}_{E_{m}^{+}-\delta}X(\theta) = B-\delta P(\theta),
\end{equation*}
where
\begin{equation}
P(\theta) = \begin{pmatrix}
X_{11}(\theta)X_{12}(\theta)-\zeta X_{11}^{2}(\theta)&-\zeta X_{11}(\theta)X_{12}(\theta) +X_{12}^{2}(\theta)\\
-X_{11}^{2}(\theta)&-X_{11}(\theta)X_{12}(\theta)
\end{pmatrix}, \label{P}
\end{equation}
with estimate
\begin{equation}
\|P(\theta)\|_{\tilde{k}} \leq (1+\zeta)\|X\|^{2}_{\tilde{k}},  \ \tilde{k}\leq [\frac{k}{400}].
\end{equation}

\begin{lemma}\label{kam}
	Suppose that $\alpha\in {\rm DC}_{d}(\gamma,\tau)$ and $P(\theta)\in C^{\tilde{k}}(\mathbb{T}^{d},{\rm SL}(2,\mathbb{R}))$ of form {\rm (\ref{P})}.
	Let $D_{\tau} = 8\sum_{m=1}^{\infty} (2\pi m)^{-(\tilde{k}-\hat{k}-3\tau-d+1)}$ with $\hat{k}\in
	\mathbb{Z}$ and $\hat{k}<\tilde{k}-3\tau-d$.
	If $0<\delta <  D_{\tau}^{-1}\gamma^{3}\|X(\theta)\|_{\tilde{k}}^{-2}$, then there exist $\widetilde{X}(\theta) \in C^{\hat{k}}(2\mathbb{T}^{d},{\rm SL}(2,\mathbb{R}))$  and $P_{1}(\theta)\in C^{\hat{k}}(\mathbb{T}^{d},{\rm gl}(2,\mathbb{R}))$ such that
	\begin{equation*}
	\widetilde{X}(\theta+\alpha)^{-1}(B-\delta P(\theta))\widetilde{X}(\theta) = \exp(b_{0}-\delta b_{1}) +\delta^{2}P_{1}(\theta),
	\end{equation*}
	where $b_{0}=\begin{pmatrix}
	0&\zeta\\
	0&0
	\end{pmatrix}$ and
	\begin{equation} \label{b}
	b_{1}=\begin{pmatrix}
	[X_{11}X_{12}]-\frac{\zeta}{2}[X_{11}^{2}] &-\zeta[X_{11}X_{12}]+[X_{12}^{2}]\\
	-[X_{11}^{2}]&-[X_{11}X_{12}]+\frac{\zeta}{2}[X_{11}^{2}]
	\end{pmatrix}
	\end{equation}
	with estimates
	\begin{align*}
	&\|\widetilde{X}(\theta)-{\rm Id}\|_{\hat{k}} \leq D_{\tau} \gamma^{-3}\delta \|X\|_{\tilde{k}}^{2},\\
	&\|P_{1}(\theta)\|_{\hat{k}} \leq 53D_{\tau}^{2}\gamma^{-6}\|X\|_{\tilde{k}}^{4}  +\delta^{-1}\zeta^{2}\|X\|_{\tilde{k}}^{2}.
	\end{align*}
\end{lemma}
\begin{proof}
	Let $G:=-\delta B^{-1}P$, one can see that $G\in C^{\tilde{k}}(\mathbb{T}^{d},{\rm sl}(2,\mathbb{R}))$.
    We first solve the linearized cohomological equation
    \begin{equation*}
    -Y(\theta+\alpha)B + BY(\theta) = B(G(\theta)-[G]).
    \end{equation*}
    Compare the Fourier coefficients of two sides,
    and by the polynomial decay of Fourier coefficients $\widehat{G}(n)$, we have
    \begin{equation*}
    \begin{split}
    \|Y(\theta)\|_{\hat{k}}
    &\leq \sum_{n\in \mathbb{Z}^{d}} \frac{\|G\|_{\tilde{k}} |2\pi n|^{-\tilde{k}}}{|e^{2\pi i\langle n,\alpha \rangle}-1|^{3}} |2\pi n|^{\hat{k}}\\
    &\leq \gamma^{-3}\|G\|_{\tilde{k}} \sum_{n\in \mathbb{Z}^{d}\backslash\{0\}} |2\pi n|^{-(\tilde{k}-\hat{k}-3\tau)} \\
    &\leq \frac{1}{4}D_{\tau} \gamma^{-3}\delta \|P\|_{\tilde{k}},
    \end{split}
    \end{equation*}
    where $D_{\tau} = 8\sum_{m=1}^{\infty} (2\pi m)^{-(\tilde{k}-\hat{k}-3\tau-d+1)} <\infty$ if $\hat{k}<\tilde{k}-3\tau-d$. Let $\widetilde{X} =e^{Y}$, we have
    \begin{equation*}
    \widetilde{X}(\theta+\alpha)^{-1}(B-\delta P(\theta)) \widetilde{X}(\theta) = Be^{[G]}+\widetilde{P}(\theta),
    \end{equation*}
    with estimate
    \begin{equation*}
    \|\widetilde{X}-{\rm Id}\|_{\hat{k}} \leq 2\|Y(\theta)\|_{\hat{k}} \leq \frac{1}{2}D_{\tau} \gamma^{-3}\delta \|P\|_{\tilde{k}} \leq D_{\tau} \gamma^{-3}\delta \|X\|_{\tilde{k}}^{2},
    \end{equation*}
    where
    \begin{equation*}
    \begin{split}
    \widetilde{P}(\theta) &= \sum_{m+n\geq 2}\frac{1}{m!}(-Y(\theta+\alpha))^{m}B\frac{1}{n!}Y(\theta)^{n}    -B\sum_{n\geq 2} \frac{1}{n!}[G]^{n}\\
    & \ \ \ \ - \delta\sum_{m+n\geq 1}\frac{1}{m!}(-Y(\theta+\alpha))^{m}P(\theta)\frac{1}{n!}Y(\theta)^{n}.
    \end{split}
    \end{equation*}
    Note that $\sum_{m+n=k}\frac{k!}{m!n!}  =2^{k}$ and $\|G\|_{\hat{k}} \leq 2\delta \|P\|_{\hat{k}}$, we have
    \begin{equation*}
    \begin{split}
    \left\|\sum_{m+n\geq 2}\frac{1}{m!}(-Y(\theta+\alpha))^{m}B\frac{1}{n!}Y(\theta)^{n} \right\|_{\hat{k}} &\leq 2\times 4 \|Y\|_{\hat{k}}^{2}\times\sum_{m+n=2}\frac{\|B\|}{m!n!}\\
    &\leq 2 D_{\tau}^{2}\gamma^{-6}\delta^{2}\|P\|^{2}_{\tilde{k}},
    \end{split}
    \end{equation*}
    \begin{equation*}
    \begin{split}
    \left\|\delta\sum_{m+n\geq 1}\frac{1}{m!}(-Y(\theta+\alpha))^{m}P(\theta)\frac{1}{n!}Y(\theta)^{n}\right\|_{\hat{k}}&\leq 2\delta\times 2\|Y\|_{\hat{k}} \times \sum_{m+n=1}\frac{\|P\|_{\hat{k}}}{m!n!}\\
    &\leq 2D_{\tau}\gamma^{-3}\delta^{2}\|P\|_{\tilde{k}}^{2},
    \end{split}
    \end{equation*}
    \begin{equation*}
    \begin{split}
    \left\|B\sum_{n\geq 2} \frac{1}{n!}[G]^{n}\right\|_{\hat{k}} &\leq 4\delta^{2}\|P\|_{\tilde{k}}^{2}.
    \end{split}
    \end{equation*}
 Hence, it follows that
    \begin{equation*}
    \|\widetilde{P}(\theta)\|_{\hat{k}}\leq 4 D_{\tau}^{2}\gamma^{-6}\delta^{2}\|P\|^{2}_{\tilde{k}}.
    \end{equation*}

    One can define $\widetilde{P}_{1}:=\delta^{-2} \widetilde{P}+\sum_{j\geq2}(j!)^{-1}(-\delta)^{j-2}B[B^{-1}P]^{j}$ such that
    \begin{equation*}
    Be^{[G]}+\widetilde{P}(\theta) = B-\delta[P] +\delta^{2}\widetilde{P}_{1}(\theta).
    \end{equation*}
    By direct calculation, we have
    \begin{equation*}
    B-\delta[P] = {\rm Id}+(b_{0}-\delta b_{1}) -\frac{\delta}{2} (b_{0}b_{1}+b_{1}b_{0}).
    \end{equation*}
    Since $b_{0}$ is nilpotent, one can check that
    \begin{equation*}
    \widetilde{X}(\theta+\alpha)^{-1}(B-\delta P(\theta))\widetilde{X}(\theta) = \exp(b_{0}-\delta b_{1}) +\delta^{2}P_{1}(\theta),
    \end{equation*}
    where $P_{1}(\theta) = \widetilde{P}_{1}-\frac{1}{2}b_{1}^{2}-\delta^{-2}\sum_{j\geq3} (j!)^{-1}(b_{0}-\delta b_{1})^{j}$ with estimate
    \begin{equation*}
    \begin{split}
    \left\|P_{1}(\theta)\right\|_{\hat{k}} &\leq \|\widetilde{P}_{1}(\theta)\|_{\hat{k}}+\frac{1}{2}\|b_{1}\|^{2} + 2\delta^{-2}\times\frac{1}{3!}\|(b_{0}-\delta b_{1})^{3}\|\\
    &\leq 4D_{\tau}^{2}\gamma^{-6}\|P(\theta)\|_{\tilde{k}}^{2} + 2\times\frac{1}{2!} \|B\|^{3} \|P(\theta)\|_{\tilde{k}}^{2} + \frac{1}{2} \|P(\theta)\|_{\tilde{k}}^{2} \\
    &\ \ \ + \frac{2}{3!}\delta^{-2}\left(\delta^{3}\|P(\theta)\|_{\tilde{k}}^{3}+3\zeta\delta^{2}\|P(\theta)\|_{\tilde{k}}^{2}+\delta\zeta^{2}\|P(\theta)\|_{\tilde{k}}\right) \\
    &\leq 8D_{\tau}^{2}\gamma^{-6}\|X(\theta)\|_{\tilde{k}}^{4} + 32\|X(\theta)\|_{\tilde{k}}^{4} +2\|X(\theta)\|_{\tilde{k}}^{4} \\
    & \ \ \ + \left(9\|X\|_{\tilde{k}}^{4}+2\|X\|_{\tilde{k}}^{4} + \delta^{-1}\zeta^{2}\|X\|_{\tilde{k}}^{2}\right)  \\
    &\leq 53D_{\tau}^{2}\gamma^{-6}\|X(\theta)\|_{\tilde{k}}^{4}  +\delta^{-1}\zeta^{2}\|X\|_{\tilde{k}}^{2},
    \end{split}
    \end{equation*}
    note that the third step use the condition $\delta <D_{\tau}^{-1}\gamma^{3}\|X(\theta)\|_{\tilde{k}}^{-2}$. Hence we finish the proof.
\end{proof}
\begin{remark}
	In the estimate of $\|P_1\|_{\hat{k}}$, we can not get rid of $\delta$ because of the non-commutative property of the matrix multiplication in general.
\end{remark}

\subsection{The upper bound of spectral gap}
With the help of Moser-P\"{o}schel argument and the reducibility of the Schr\"{o}dinger cocycle, we are able to prove the first main theorem.
\begin{proof}[Proof of Theorem \ref{main}]

Rewrite Schr\"{o}dinger cocycle $(\alpha, S^{V}_{E}) $ as  $(\alpha, A_{E}+F(\theta))$, where
\begin{equation*}
A_{E} = \begin{pmatrix}
E&-1\\
1&0
\end{pmatrix},\ \ \ F(\theta) = \begin{pmatrix}
-V(\theta+n\alpha)&0\\
0&0	
\end{pmatrix}.
\end{equation*}
By the assumption on $\|V\|_{k}$, we have $\|A_{E}\|\leq 3$ (so does $A_{E}^{-1}$) since  $E\leq 2+\sup_{\theta\in \mathbb{T}^{d}}\|V(\theta)\|_{\mathbb{T}^{d}}\leq 2+\varepsilon$. The norm of $A_{E}$ is bounded uniformly with respect to $E$. If we write
\begin{equation*}
A_{E} +F(\theta)= A_{E}e^{f(\theta)},
\end{equation*}
with $f(\theta)\in C^{k}(\mathbb{T}^{d},{\rm sl}(2,\mathbb{R}))$, then according to Theorem \ref{ar}, one can obtain $\varepsilon=\varepsilon(\gamma,\tau,k,d)$ which is independent of $E$, such that if $\|V\|_{k}\leq\varepsilon$ (the assumption of Theorem \ref{ar} are naturally fulfilled as $\|f\|_{k}\leq 2\|A_{E}^{-1}\|\|F\|_{k}$), then
the cocycle $(\alpha, S^{V}_{E}(\theta))$ is $C^{k,\tilde{k}}$ almost reducible. Moreover, if $\rho(\alpha, S^{V}_{E_{m}^{+}}) = \frac{\langle m,\alpha \rangle}{2}$ for $m\in \mathbb{Z}^{d}\backslash\{0\}$, by Theorem \ref{r} we have
\begin{equation}\label{conjugate}
X(\theta+\alpha)^{-1}S^{V}_{E_{m}^{+}}X(\theta) = \begin{pmatrix} 1&\zeta\\0&1
\end{pmatrix},
\end{equation}
with $\zeta \leq \varepsilon^{\frac{1}{3}}|m|^{-\frac{k}{8.5}}$ and $\|X(\theta)\|_{\tilde{k}}\leq D_{1} |m|^{(2\tilde{k}+\tau)}$, where $ \tilde{k}\leq[\frac{k}{400}]$ and $D_{1} = D_{1}(\gamma, \tau, d,k)$.
From now on, we fix $\tilde{k}=[\frac{k}{5000}]$, then for any $m\in \mathbb{Z}^{d}\backslash\{0\}$ we have
\begin{equation}\label{tiaojian}
\|X\|_{\tilde{k}}^{14}\zeta^{\frac{1}{18}}\leq D_{1}^{14} |m|^{14(2\tilde{k}+\tau)} \varepsilon^{\frac{1}{54}} |m|^{-\frac{17}{36}k}\leq 10^{-5}D_{\tau}^{-4}\gamma^{12}.
\end{equation}
The above inequality is possible since one can choose $\varepsilon$ sufficiently small and the smallness only depend on $\gamma,\tau, k,d$.

For any $\delta\in(0,1)$, we define a function $d(\delta):=\det(b_{0}-\delta b_{1}) + \frac{1}{4}\delta^{2}\zeta^{2}[X_{11}^{2}]^{2}$, where $b_{0}$ and $b_{1}$ are defined in (\ref{b}). By a direct calculation, one can get that
\begin{equation}
\begin{split}
d(\delta) & = -\delta [X_{11}^{2}]\zeta +\delta^{2}([X_{11}^{2}][X_{12}^{2}]-[X_{11}X_{12}]^{2})\\
& = \delta ([X_{11}^{2}][X_{12}^{2}]-[X_{11}X_{12}]^{2})\left(\delta - \frac{[X_{11}^{2}]\zeta }{[X_{11}^{2}][X_{12}^{2}]-[X_{11}X_{12}]^{2}} \right).
\end{split} \label{d}
\end{equation}
To further estimate $X$ we recall the following fundamental lemma which has been proved in $C^{\omega}$ case in \cite{Le}, however it holds for $C^{k}$ case just by replacing the analytic norm.
\begin{lemma}[\cite{Le}] \label{poly}
	Let $X\in C^{\tilde{k}}(2\mathbb{T}^{d},{\rm SL}(2,\mathbb{R}))$ satisfying (\ref{conjugate}), then for any $\kappa \in (0,\frac{1}{4})$,  if $\|X\|_{\tilde{k}}\zeta^{\frac{\kappa}{2}}\leq \frac{1}{4}$, the followings hold:
	\begin{align*}
	&0<\frac{[X_{11}^{2}]}{[X_{11}^{2}][X_{12}^{2}]-[X_{11}X_{12}]^{2}} \leq \frac{1}{2}\zeta^{-\kappa},\\
	&[X_{11}^{2}][X_{12}^{2}]-[X_{11}X_{12}]^{2} \geq 8\zeta^{2\kappa}.
	\end{align*}
\end{lemma}
Let $\delta_{1} = \zeta^{\frac{17}{18}}$. By (\ref{tiaojian}) and  $\|X\|_{\tilde{k}}\geq 1$, we have
\begin{equation*}
\delta_{1}D_{\tau} \gamma^{-3} \|X(\theta)\|_{\tilde{k}}^{2}\leq \zeta^{\frac{1}{52}} D_{\tau} \gamma^{-3}  \|X(\theta)\|_{\tilde{k}}^{\frac{7}{2}} \leq 10^{-\frac{5}{4}}<1,
\end{equation*}
which deduces that $0<\delta_{1}<D_{\tau}^{-1}\gamma^{3}\|X(\theta)\|_{\tilde{k}}^{-2}$. Then by Lemma \ref{kam}, there exist $\widetilde{X}\in C^{\hat{k}}(2\mathbb{T}^{d},{\rm SL}(2,\mathbb{R}))$ and $P_{1}\in C^{\hat{k}}(\mathbb{T}^{d},{\rm gl}(2,\mathbb{R}))$ such that the cocycle $(\alpha, B-\delta_{1}P(\theta))$ is conjugated to $(\alpha, e^{b_{0}-\delta_{1}b_{1}}+\delta_{1}^{2}P_{1})$ by $\widetilde{X}$. Since $\widetilde{X}$ is homotopic to identity by construction, we have
\begin{equation*}
\rho(\alpha, B-\delta_{1}P(\theta)) =\rho(\alpha, e^{b_{0}-\delta_{1}b_{1}}+\delta_{1}^{2}P_{1}(\theta)).
\end{equation*}
To prove that $|G(V)|\leq \delta_{1}$, it is sufficient to show that $\rho(\alpha, e^{b_{0}-\delta_{1}b_{1}}+\delta_{1}^{2}P_{1}(\theta))>0$ by monotonicity of rotation number. According to (\ref{tiaojian}), one can check that $
\|X\|_{\tilde{k}}\zeta^{\frac{1}{36}} \leq 10^{-\frac{5}{2}}D_{\tau}^{-2} \|X\|^{-6}_{\tilde{k}} \gamma^{6} \leq \frac{1}{4}$.
Apply Lemma \ref{poly} to (\ref{d}), for $d(\delta_{1}) = \det(b_{0}-\delta_{1}b_{1})+\frac{1}{4}\delta_{1}^{2}\zeta^{2}[X_{11}^{2}]^{2}$, we have
\begin{equation*}
d(\delta_{1}) \geq \zeta^{\frac{17}{18}}\times 8\zeta^{\frac{1}{9}}\times \frac{1}{2}\zeta^{\frac{17}{18}} =4\zeta^{2}.
\end{equation*}
Moreover, by (\ref{tiaojian}) it is easy to see that
\begin{equation}
\begin{split}
\det (b_{0}-\delta_{1}b_{1})
\geq 4\zeta^{2}-\frac{1}{4}\delta_{1}^{2}\zeta^{2}[X_{11}^{2}]^{2}
\geq 4\zeta^{2}(1-\frac{1}{16}\zeta^{\frac{17}{9}}\|X\|_{\tilde{k}}^{4})\geq 3\zeta^{2}.
\end{split} \label{c1}
\end{equation}
Hence, there exists $\mathcal{P}\in {\rm SL}(2,\mathbb{R})$ such that
\begin{equation*}
\mathcal{P}^{-1}e^{b_{0}-\delta_{1}b_{1}}\mathcal{P} = \exp \begin{pmatrix}
0&\sqrt{\det (b_{0}-\delta_{1}b_{1})}\\
-\sqrt{\det (b_{0}-\delta_{1}b_{1})}&0
\end{pmatrix}:=\Delta
\end{equation*}
with $
\|\mathcal{P}\|\leq 2 \left(\frac{\|b_{0}-\delta_{1}b_{1}\|}{\sqrt{\det (b_{0}-\delta_{1}b_{1})}}\right)^{\frac{1}{2}}$. Since $\|b_{0}-\delta_{1}b_{1}\|\leq \zeta+\delta_{1}(1+\zeta)|X|_{\mathbb{T}^{d}}^{2}\leq \frac{3}{2}\zeta^{\frac{17}{18}} \|X\|^{2}_{\tilde{k}}$, we have
\begin{equation*}
\begin{split}
\frac{\|b_{0}-\delta_{1}b_{1}\|}{\sqrt{\det (b_{0}-\delta_{1}b_{1})}}
&\leq \frac{\frac{3}{2}\zeta^{\frac{17}{18}}\|X\|^{2}_{\tilde{k}}}{\sqrt{3}\zeta} \leq \|X\|_{\tilde{k}}^{2}\zeta^{-\frac{1}{18}}.
\end{split}
\end{equation*}
According to Lemma \ref{eig} and Lemma \ref{kam} with
$\mathcal{P}^{-1} (e^{b_{0}-\delta_{1}b_{1}}+\delta_{1}^{2}P_{1})\mathcal{P} = \Delta  +\mathcal{P}^{-1}\delta_{1}^{2}P_{1}(\theta)\mathcal{P}$, we have
\begin{equation}
\begin{split}
&\ \ \ \ |\rho(\alpha,e^{b_{0}-\delta_{1}b_{1}}+\delta_{1}^{2}P_{1})-\sqrt{\det(b_{0}-\delta_{1}b_{1})}| \\
&= |\rho(\alpha,\Delta  +\mathcal{P}^{-1}\delta_{1}^{2}P_{1}(\theta)\mathcal{P} )-\rho(\alpha, \Delta)|\\
&\leq \delta_{1}^{2}\|\mathcal{P}\|^{2} \|P_{1}\|_{\hat{k}}\\
&\leq \zeta^{\frac{17}{9}}\times 4 \|X\|_{\tilde{k}}^{2}\zeta^{-\frac{1}{18}} \times ( 53D_{\tau}^{2}\gamma^{-6}\|X\|_{\tilde{k}}^{4} +\zeta^{-\frac{17}{18}}\zeta^{2}\|X\|_{\tilde{k}}^{2} )\\
&\leq 240 D_{\tau}^{2}\gamma^{-6}\|X\|_{\tilde{k}}^{6}\zeta^{\frac{11}{6}}.
\end{split} \label{c2}
\end{equation}
By the assumption (\ref{tiaojian}), it deduces that
\begin{equation*}
\begin{split}
160D_{\tau}^{2}\gamma^{-6}\|X\|_{\tilde{k}}^{6}\zeta^{\frac{5}{6}} &\leq 160D_{\tau}^{2}\gamma^{-6}\|X\|_{\tilde{k}}^{6}\zeta^{\frac{5}{6}} \times (10^{-5}D_{\tau}^{-4}\gamma^{12}\|X\|_{\tilde{k}}^{-14}\zeta^{-\frac{1}{18}})\\
&\leq 160\times 10^{-5}D_{\tau}^{-2}\gamma^{6}\|X\|_{\tilde{k}}^{-8} \zeta^{\frac{7}{9}}\\
&<1,
\end{split}
\end{equation*}
then combine with (\ref{c1}) and  (\ref{c2}), we have
\begin{equation*}
\begin{split}
\rho(\alpha, e^{b_{0}-\delta_{1}b_{1}}+\delta_{1}^{2}P_{1})&\geq |\rho(\alpha,\Delta)|- |\rho(\alpha,\Delta  +\mathcal{P}^{-1}\delta_{1}^{2}P_{1}(\theta)\mathcal{P} )-\rho(\alpha, \Delta)|\\
& \geq \sqrt{3}\zeta -160D_{\tau}^{2}\gamma^{-6}\|X\|_{\tilde{k}}^{6}\zeta^{\frac{11}{6}}\\
&\geq \sqrt{3}\zeta-\frac{3}{2}\zeta >0.
\end{split}
\end{equation*}
Hence, by $\zeta \leq \varepsilon^{\frac{1}{3}}|m|^{-\frac{k}{8.5}}$ we have
\begin{equation*}
|G_{m}(V)|\leq \delta_{1}=\zeta^{\frac{17}{18}} \leq\varepsilon^{\frac{1}{4}}|m|^{-\frac{k}{9}}, \ \ \ \forall \ m\in \mathbb{Z}^{d}\backslash\{0\}.
\end{equation*}
This finishes the proof of Theorem \ref{main}.
\end{proof}

\section{Homogeneous spectrum}

As stated, homogeneous spectrum follows by polynomial decay of gap length and H\"{o}lder continuity of IDS. In this final section, we prove our second main theorem.
\begin{proof}[Proof of Theorem 1.2]
	Consider two different gaps $G_{m}(V)$ and $G_{m'}(V)$, without loss of generality, we assume that $E_{m}^{+}\leq E_{m'}^{-}$.
	Hence, one can obtain that $ {\rm dist} ( G_{m}(V), G_{m'} (V) ) = E_{m'}^{-} - E_{m}^{+} $. Set $\underline{E}=\min \Sigma_{V,\alpha}$ $\overline{E}=\max \Sigma_{V,\alpha}$ and $G_{0}(V)=(-\infty,\underline{E})\cup (\overline{E},+\infty)$.
	We need the following lemma.
	\begin{lemma}[\cite{Cai2}]\label{holder}
		Let $\alpha \in {\rm DC}_{d}(\gamma,\tau)$, $V\in C^{k}(\mathbb{T}^{d},\mathbb{R})$ with $k\geq 5D\tau$, where $D$ is a numerical constant. Then there exists $\varepsilon_{5}=\varepsilon_{5}(\gamma,\tau,k,d)$ such that if $\|V\|_{k}\leq \varepsilon_{5}$, then $N=N_{V,\alpha}$ is $\frac{1}{2}$-H\"{o}lder continuous:
		\begin{equation*}
		N(E+\tilde{\epsilon})-N(E-\tilde{\epsilon}) \leq C_{0}\tilde{\epsilon}^{\frac{1}{2}}, \ \ \forall \ \tilde{\epsilon}>0, \ \forall \ E\in \mathbb{R},
		\end{equation*}
		where $C_{0}=C_{0}(\gamma,\tau,d)$.
	\end{lemma}

	Let $E=\frac{1}{2}(E_{m'}^{-}+E_{m}^{+})$ and $\tilde{\epsilon} = \frac{1}{2}(E_{m'}^{-}-E_{m}^{+})$, by Lemma \ref{holder} we have
\begin{equation}
|N(E_{m'}^{-})-N(E_{m}^{+})|\leq C_{0}(E_{m'}^{-}-E_{m}^{+})^{\frac{1}{2}}, \label{m1}
\end{equation}
where $C_{0}$ is independent of $E$. Since $\alpha \in {\rm DC}_{d}(\gamma,\tau)$,  according to (\ref{ids}), we also have
\begin{equation}
|N(E_{m'}^{-})-N(E_{m}^{+})|\geq \|\langle m'-m,\alpha \rangle\|_{\mathbb{T}^{d}} \geq \frac{\gamma}{|m'-m|^{\tau}}.	\label{m2}
\end{equation}
Then by (\ref{m1}) and (\ref{m2}) we conclude that
\begin{equation}
{\rm dist}(G_{m}(V),G_{m'}(V)) \geq (\frac{\gamma}{C_{0}})^{2}\cdot |m'-m|^{-2\tau}, \ \ \forall \ m'\neq m \in \mathbb{Z}^{d}. \label{dist}
\end{equation}
One can use the same way to show that
\begin{equation}
|E_{m}^{-}-\underline{E}| \geq (\frac{\gamma}{C_{0}})^{2}\cdot |m|^{-2\tau}, \ \  |E_{m}^{+}-\overline{E}| \geq (\frac{\gamma}{C_{0}})^{2}\cdot |m|^{-2\tau}  \label{a}.
\end{equation}

Given any $E\in \Sigma_{V,\alpha}$ and any $\epsilon>0$, we define
\begin{equation*}
\mathcal{M}=\mathcal{M}(E,\epsilon):=\{m\in \mathbb{Z}^{d}\backslash \{0\}:G_{m}(V)\cap(E-\epsilon,E+\epsilon)\neq \emptyset\},
\end{equation*}
and let $m_{0} \in \mathcal{M} $ be such that $|m_{0}| = \min_{m\in \mathcal{M}}|m|$.
Since $E\in \Sigma_{V,\alpha}$, it is obvious that
\begin{equation}
\begin{split}
&|G_{m_{0}}(V)\cap (E-\epsilon,E+\epsilon)|\leq \epsilon,\\
&|(-\infty,\underline{E})\cap (E-\epsilon,E+\epsilon)|\leq \epsilon,\\
&|(\overline{E},+\infty)\cap (E-\epsilon,E+\epsilon)|\leq \epsilon.
\end{split} \label{cases}
\end{equation}

{\bf Case 1.} $G_{0}(V)\cap(E-\epsilon,E+\epsilon)=\emptyset$.
By the definition of $\mathcal{M}$, we have
$
{\rm dist}(G_{m}(V),G_{m_{0}}(V))\leq 2\epsilon, \ \ \forall \ m\in \mathcal{M}$.
On the other hand, by (\ref{dist}),
\begin{equation*}
{\rm dist}(G_{m}(V),G_{m_{0}}(V))\geq (\frac{\gamma}{C_{0}})^{2}|2m|^{-2\tau}, \ \ \ \forall \ m \in \mathcal{M}\backslash\{m_{0}\}.
\end{equation*}
Therefore, we get $|m|\geq C_{1}\epsilon^{-\frac{1}{2\tau}}$ with $C_{1} = C_{1}(\gamma,\tau,C_{0})$. From Theorem \ref{main}, we have $|G_{m}(V)| \leq \varepsilon^{\frac{1}{4}} |m|^{-\frac{k}{9}}$, hence by direct calculation
	\begin{equation*}
\begin{split}
\sum_{m\in \mathcal{M}\backslash\{m_{0}\}}|G_{m}(V)\cap(E-\epsilon,E+\epsilon)|&\leq \sum_{m\in \mathcal{M}\backslash\{m_{0}\}} E_{m}^{+}-E_{m}^{-} \\
&\leq  \sum_{|m|\geq C_{1}\epsilon^{-\frac{1}{2\tau}}}\varepsilon^{\frac{1}{4}}|m|^{-\frac{k}{9}} \\
&\leq \frac{\epsilon}{4},
\end{split}
\end{equation*}
provided $\epsilon \leq \epsilon_{1}$ with $\epsilon_{1}=(\frac{1}{8}C_{1}^{\frac{k}{9}}\varepsilon^{-\frac{1}{4}})^{\frac{18\tau}{k-18\tau}}$. Combine with (\ref{cases}), we deduce that
\begin{equation*}
\begin{split}
& \ \ \ \ |(E-\epsilon,E+\epsilon)\cap \Sigma_{V,\alpha}|\\
&\geq 2\epsilon-|G_{m_{0}}(V)\cap(E-\epsilon,E+\epsilon)|-\sum_{m\in \mathcal{M}\backslash\{m_{0}\}}|G_{m}(V)\cap(E-\epsilon,E+\epsilon)|\\
&\geq\frac{3}{4}\epsilon, \ \ \ \forall \ 0 < \epsilon\leq \epsilon_{1}.
\end{split}
\end{equation*}

{\bf Case 2.} $(-\infty,\underline{E})\cap(E-\epsilon,E+\epsilon)\neq\emptyset$. For any $m\in \mathcal{M}$, we have
$|E_{m}^{-}-\underline{E}|\leq 2\epsilon$,
then combine with (\ref{a}), one can get
$|m|\geq 2C_{1}(\gamma,\tau,C_{0})\epsilon^{-\frac{1}{2\tau}}$. Just as Case 1, we may conclude that
\begin{equation*}
\sum_{m\in \mathcal{M}}|G_{m}(V)\cap(E-\epsilon,E+\epsilon)| \leq \sum_{|m|\geq 2C_{1}\epsilon^{-\frac{1}{2\tau}}}\varepsilon^{\frac{1}{4}}|m|^{-\frac{k}{9}}\leq \frac{\epsilon}{4}.
\end{equation*}
provided $\epsilon \leq \epsilon_{2}$ with $\epsilon_{2}=(2^{\frac{k}{9}-3}C_{1}^{\frac{k}{9}}\varepsilon^{-\frac{1}{4}})^{\frac{18\tau}{k-18\tau}}$. So we have
\begin{equation*}
\begin{split}
& \ \ \ \ |(E-\epsilon,E+\epsilon)\cap \Sigma_{V,\alpha}|\\	
&\geq 2\epsilon-|(-\infty,\underline{E})\cap(E-\epsilon,E+\epsilon)|-\sum_{m\in \mathcal{M}}|G_{m}(V)\cap(E-\epsilon,E+\epsilon)|\\
&\geq\frac{3}{4}\epsilon, \ \ \ \forall \ 0 < \epsilon\leq \epsilon_{2}.
\end{split}
\end{equation*}

{\bf Case 3.}
$(\overline{E},+\infty,)\cap(E-\epsilon,E+\epsilon)\neq\emptyset$. The proof is similar to the {\bf Case 2}, one can choose $\epsilon_{3}=\epsilon_{2}$ to finish the proof.
	
Finally, let $\epsilon_{0}= \min\{\epsilon_{1},\epsilon_{2},\epsilon_{3}\}$, we get
\begin{equation*}
	|(E-\epsilon,E+\epsilon)\cap \Sigma_{V,\alpha}| \geq \frac{3}{4}\epsilon, \ \ \ \forall \ E\in \Sigma_{V,\alpha}, \ \  \forall \ 0 < \epsilon\leq \epsilon_{0}.
	\end{equation*}
	 As for the case $\epsilon\in (\epsilon_{0},{\rm diam}{\Sigma_{V,\alpha}})$, we have
	 \begin{equation*}
	 |(E-\epsilon,E+\epsilon)\cap {\Sigma_{V,\alpha}}| \geq \frac{3}{4} \epsilon_{0} \geq \frac{3\epsilon_{0}}{4\  {\rm diam}{\Sigma_{V,\alpha}}} \times \epsilon.
	 \end{equation*}
	Choose $\mu = \min \{\frac{3}{4},\frac{3\epsilon_{0}}{4\  {\rm diam}{\Sigma_{V,\alpha}}}\}$ and this concludes the proof.
\end{proof}

\section*{Acknowledgments}
The authors want to thank Jiangong You and Qi Zhou for useful discussions. Ao Cai would also like to thank Pedro Duarte for his consistent support at University of Lisbon. This work is supported by PTDC/MAT-PUR/29126/2017, Nankai Zhide Fundation and NSFC grant (11671192).


\begin{thebibliography}{99}
	\bibitem{Av2}Avila, A., Bochi, J., Damanik, D.: {\it  Cantor spectrum for Schr\"{o}dinger operators with potentials arising from generalized skew-shifts}, Duke. Math. J., \textbf{146}, 253-280 (2009).
	
	\bibitem{Av1}Avila, A., Jitomirskaya, S.: {\it The ten Martini problem}, Ann. Math., \textbf{170}, 303-342 (2009).

	\bibitem{AJ2}Avila, A., Jitomirskaya, S.: {\it Almost localization and almost reducibility}, J. Eur. Math. Soc., \textbf{12}, 93-131 (2010).

	\bibitem{Av2}Avila, A., You, J., Zhou, Q.: {\it Dry ten Martini problem in the non-critical case}, preprint.
	
	\bibitem{Av3}Avila, A., Last, Y., Shamis, M., Zhou, Q.: {\it On the abominable properties of the almost Mathieu operator with well approximated frequencies}, preprint.

	\bibitem {AS}Avron, J., Simon, B.: {\it Almost periodic Schr\"{o}dinger operators II, the integrated density of states}, Duke. Math. J., \textbf{506}, 369-390 (1983).

	\bibitem{BS} Bellissard, J., Simon, B.: {\it Cantor spectrum for the almost Mathieu equation}, J. Funct. Anal., \textbf{48}, 408-419 (1982).

	\bibitem{Bin}Binder, I., Damanik, D., Goldstein, M., Lukic, M.: {\it Almost periodicity in time of solutions of the KdV equation}, arXiv: 1509.07373.
	
	\bibitem{Cai}Cai, A., Ge, L.: {\it Reducibility of finitely differentiable quasi-periodic cocycles and its spectral applications}, arXiv:1712.09041.

	\bibitem{Cai2} Cai, A., Chavaudret, C., You, J., Zhou, Q.: {\it Sharp H\"{o}lder continuity of the Lyapunov exponent of finitely differentiable quasi-periodic cocycles}, Math. Z., \textbf{291}, 931-958 (2019).
		
	\bibitem{Carl}Carleson, L.: {\it On $H^{\infty}$ in multiply connected domains}, in {\it Conference on harmonic analysis in honor of Antoni Zygmund, Vol. $\Rmnum{1}$, \textbf{\Rmnum{2}} (Chicago, Ill., 1981)}, Wadsworth Math. Ser., 349-372. Wadsworth, Belmont, CA, (1983).

	\bibitem{Ds}Damanik, D.: {\it Schr\"{o}dinger operators with dynamically defined potentials}, Ergod. Th. \& Dynam. Sys.,  \textbf{37}, 1681-1764 (2017).
	
	\bibitem{Dam}Damanik, D., Goldstein, M.: {\it On the inverse spectral problem for the quasi-periodic Schr\"{o}dinger
		equation}, Publ. Math. Inst. Hautes \'{E}tudes Sci., \textbf{119}, 217-401 (2014).

	\bibitem{Dam2} Damanik, D., Goldstein, M., Lukic, M.: {\it  The spectrum of a Schr\"{o}dinger operator with small quasi-periodic potential is homogeneous}, J. Spectr. Theory, \textbf{6}, 415-427 (2016).

	\bibitem{Dam3}Damanik, D., Goldstein, M., Lukic, M.: {\it The isospectral torus of quasi-periodic Schr\"{o}dinger operators via periodic approximations}, Invent. Math., \textbf{207}, 895-980 (2017).

	\bibitem{DKL}Damanik, D., Killip, R., Lenz, D.: {\it Uniform spectral properties of one-dimensional quasicrystals. iii, $\alpha$-continuity}, Commun. Math. Phys., \textbf{212}, 191-204 (2000).
	
	\bibitem{De} Delyon, F., Souillard, B.: {\it The rotation number for finite difference operators and its properties}, Commun. Math. Phys., \textbf{89}, 415-426 (1983).
	
	\bibitem{Eli}Eliasson, H.: {\it Floquet solutions for the 1-dimensional quasi-periodic Schr\"{o}dinger equation}, Commun. Math. Phys., \textbf{146}, 447-482 (1992).
		
	\bibitem{Gold}Goldstein, M., Schlag, W., Voda, M.: {\it On the spectrum of multi-frequency quasiperiodic Schr\"{o}dinger operators with large coupling}, arXiv:1708.09711.
	
	\bibitem{Am}Hadj Amor, S.: {\it H\"{o}lder continuity of the rotation number for quasi-periodic cocycles in ${\rm SL}(2,\mathbb{R})$}, Commun. Math. Phys., \textbf{187}, 565-588 (2009).


    \bibitem{Hofs}Hofstadter, D.; {\it Energy levels and wave functions of Bloch electrons in rational and irrational magnetic Telds}, Phys. Rev. B, \textbf{14}, 2239-2249 (1976).


	\bibitem{Hou}Hou, X., You, J.: {\it Almost reducibility and non-perturbative reducibility of quasi-periodic linear
	systems}, Invent. Math., \textbf{190}, 209-260 (2012).

    \bibitem{JianShi}Jian, W., Shi, Y.: {\it Sharp H\"{o}lder continuity of the integrated density of states for the extended Harper's model with a Liouville frequency}, Acta. Math. Sci., \textbf{39}, 1240-1254 (2019).

	\bibitem{Jito}Jitomirskaya, S.: {\it Almost everything about the almost Mathieu operator, II},  ``Proceedings of XI International Congress of Mathematical Physics", Int. Press, 373-382 (1995).

    \bibitem{Jo}Johnson, R.: {\it Exponential dichotomy, rotation number,  and linear differential operators with bounded coefficients}, J. Differ. Equations, \textbf{61}, 54-78 (1986).

	\bibitem{JM}Johnson, R., Moser, J.: {\it The rotation number for almost periodic potentials}, Commun. Math. Phys., \textbf{84}, 403-438 (1982).
	
	\bibitem{Le}Leguil, M., You, J., Zhao, Z., Zhou, Q.: {\it Asymptotics of spectral gaps of quasi-periodic Schr\"{o}dinger operators}, arXiv:1712.04700.
	
	\bibitem{LiuYuan}Liu, W., Yuan, X.: {\it Spectral gaps of almost Mathieu operators in the exponential regime}, J. Fractal Geom., \textbf{2}, 1-51 (2015).
	
	\bibitem{LiuShi}Liu, W., Shi, Y.: {\it Upper bounds on the spectral gaps of quasi-periodic Schr\"{o}dinger operators with Liouville frequencies}, arXiv:1708.01760.
	
	\bibitem{Mos}Moser, J., P\"{o}schel, J.: {\it An extension of a result by Dinaburg and Sinai on quasi-periodic potentials}, Commun. Math. Helv., \textbf{59}, 39-85 (1984).

	\bibitem{Puig}Puig, J.: {\it Cantor spectrum for the almost Mathieu operator}, Commun. Math. Phys., \textbf{244}, 297-309 (2004).
	
	\bibitem{Puig2}Puig, J.: {\it A nonperturbative Eliasson’s reducibility theorem}, Nonlinearity, \textbf{19}, 355-376 (2006).
	
	\bibitem{ShiYuan}Shi, Y., Yuan, X.: {\it Exponential decay of the lengths of the spectral gaps for the Extended Harper’s Model with a Liouvillean frequency}, J. Dyn. Diff. Equat., \textbf{31}, 1921-1953 (2019).
	
	\bibitem{Si}Simon, B.: {\it Almost periodic Schr\"{o}dinger operators: A review}, Adv. Appl. Math., \textbf{3}, 463-490 (1982).

	\bibitem{So1}Sodin, M., Yuditskii, P.: {\it Almost periodic Sturm-Liouville operators with Cantor homogeneous spectrum}, Comment. Math. Helv., \textbf{70}, 639-658 (1995).
	
	\bibitem{So2}Sodin, M., Yuditskii, P.: {\it Almost periodic Jacobi matrices with homogeneous spectrum, infinite-dimensional Jacobi inversion, and Hardy spaces of character-automorphic functions}, J. Geom. Anal., \textbf{7}, 387-435 (1997).
	
	\bibitem{Th}Thouless, D., Kohmoto, M., Nightingale, M.,  Den Nijs, M.: {\it Quantised Hall conductance in a two dimensional periodic potential},  Phys. Rev. Lett., \textbf{49}, 405-408 (1982).

	
	\bibitem{Wang}Wang, Y., Zhang, Z.: {\it Cantor spectrum for a class of $C^2$ quasiperiodic Schr\"{o}dinger operators}, Int. Math. Res. Not., \textbf{2017}, 2300-2336 (2017).
	
	\bibitem{Xuxu}Xu, X., Zhao, X.: {\it Exponential upper bounds on the spectral gaps and homogeneous spectrum for the non-critical extended Harper's model}, Discrete \& Cont. Dyn. Syst., \textbf{40}, 4777-4800 (2020).
	
	\bibitem{Zen}Zehnder, E.: {\it Generalized implicit function theorems with application to some small divisor problems, \Rmnum{1}},  Commun. Pure. Math., \textbf{XXVIII}, 91-140 (1975).
\end{thebibliography}
\end{document}